\pdfmapfile{=<name>.map}
\documentclass[12pt,oneside,reqno,fleqn]{amsart}
\usepackage{amssymb}
\usepackage{bbm}
\usepackage{cases}
\usepackage{amsmath}
\usepackage{graphicx}
\usepackage{mathrsfs}
\usepackage{stmaryrd}
\usepackage{color}
\usepackage{soul}
\usepackage[dvipsnames]{xcolor}
\usepackage{amsfonts}
\usepackage{amsmath,amssymb,amsthm}
\usepackage{libertine}
\usepackage[T1]{fontenc}
\usepackage[libertine,varbb]{newtxmath}\usepackage[colorlinks,linkcolor=red,citecolor=red,urlcolor=red]{hyperref}
\usepackage[shortlabels] {enumitem}
\usepackage[nameinlink,capitalize]{cleveref}
\usepackage[spacing=true,kerning=true,tracking=true,letterspace=50]{microtype}
\hypersetup{pdftitle={Stability}, pdfauthor={LG, CL}}
\numberwithin{equation}{section}
\newcommand{\be}{\begin{eqnarray}}
\newcommand{\ee}{\end{eqnarray}}
\newcommand{\ce}{\begin{eqnarray*}}
\newcommand{\de}{\end{eqnarray*}}
\newtheorem{theorem}{Theorem}[section]
\newtheorem{lemma}[theorem]{Lemma}
\newtheorem{proposition}[theorem]{Proposition}
\newtheorem{corollary}[theorem]{Corollary}
\theoremstyle{remark}
\newtheorem{assumption}[theorem]{Assumption}

\newtheorem{example}[theorem]{Example}
\newtheorem{remark}[theorem]{Remark}
\newtheorem{definition}[theorem]{Definition}

\crefname{eqn}{Equation}{Equations}
\crefname{assumption}{Assumption}{Assumptions}
\crefname{innercustomthm}{Condition}{Conditions}
\crefrangelabelformat{innercustomthm}{#3#1#4-#5#2#6}

\def\eps{\varepsilon}

\def\d{\delta}

\def\<{{\langle}}
\def\>{{\rangle}}
\def\({{\Big(}}
\def\){{\Big)}}

\def\bx{{\mathbf{x}}}

\def\={&\!\!=\!\!&}
\def\bt{\begin{theorem}}
\def\et{\end{theorem}}
\def\bl{\begin{lemma}}
\def\el{\end{lemma}}
\def\br{\begin{remark}}
\def\er{\end{remark}}
\def\bd{\begin{definition}}
\def\ed{\end{definition}}
\def\bp{\begin{proposition}}
\def\ep{\end{proposition}}
\def\bc{\begin{corollary}}
\def\ec{\end{corollary}}
\def\bx{\begin{example}}
\def\ex{\end{example}}

\def\cL{{\mathcal L}}

\def\cP{{\mathcal P}}

\def\mE{{\mathbb E}}

\def\mI{{\mathbb I}}

\def\mP{{\mathbb P}}

\def\geq{\geqslant}
\def\leq{\leqslant}

\def\div{\mathord{{\rm div}}}

\newcommand{\dd}{\,\mathrm{d}}
\newcommand{\R}{{\mathbb R}}

\newcommand{\norm}[1]{{\left\vert\kern-0.25ex\left\vert\kern-0.25ex\left\vert #1
    \right\vert\kern-0.25ex\right\vert\kern-0.25ex\right\vert}}

\allowdisplaybreaks

\begin{document}
	\title{Stability estimates for singular SDEs and applications}
	\date{\today}
	\author{Lucio Galeati \and Chengcheng Ling}
\address{Lucio Galeati:
Rheinische Friedrich-Wilhelms-Universität Bonn,
Institute for Applied Mathematics \& Hausdorf Center for Mathematics,
Endenicher Allee 60, 53115 Bonn, Germany
\\
Email:  lucio.galeati@iam.uni-bonn.de
}

\address{Chengcheng Ling:
Technische Universit\"at Berlin,
Fakult\"at II, Institut f\"ur Mathematik,
10623 Berlin, Germany
\\
Email: ling@math.tu-berlin.de
 }

	\begin{abstract}
	We consider multidimensional SDEs with singular drift $b$ and Sobolev diffusion coefficients $\sigma$, satisfying Krylov--R\"ockner type assumptions.
	We prove several stability estimates, comparing solutions driven by different $(b^i,\sigma^i)$, both for It\^o and Stratonovich SDEs, possibly depending on negative Sobolev norms of the difference $b^1-b^2$.
	We then discuss several applications of these results to McKean--Vlasov SDEs, criteria for strong compactness of solutions and Wong--Zakai type theorems.\\[1ex]
	\\[1ex]
	\end{abstract}

\maketitle	

\section{Introduction}
In this paper we consider multidimensional singular SDEs of the form
\begin{equation}\label{eq:sde}
\dd X_t = b_t(X_t)\dd t + \sigma_t(X_t)\dd W_t, \quad X\big\vert_{t=0}=X_0
\end{equation}
where $b:[0,T]\times \R^d\to \R^d$, $\sigma:[0,T]\times \R^d\to \R^{d\times d}$ and $W$ is a $d$-dimensional Brownian motion. The interval $[0,T]$ is finite, although arbitrarily large, and the initial condition $X_0$ is possibly random and independent of $W$. The drift $b$ is integrable, possibly unbounded; $b$ and $\sigma$ will be always assumed to satify relatively standard assumptions under which strong well-posedness is known to hold.
To properly formulate them, let us first define for $\beta \in [0,1]$ the sets
\begin{equation*}
\mathcal{J}_\beta:=\left\{(p,q): p,q\in \Big(\frac{2}{2-\beta},\infty\Big),\quad \frac{2}{q}+\frac{d}{p}<2-\beta\right\};
\end{equation*}
we will be mostly interested in considering $\mathcal{J}_0$ and $\mathcal{J}_1$.

\begin{assumption}\label{main-ass}
The following hold.
\begin{itemize}
    \item[($\mathcal{H}^b$)] There exist $(p_1,q_1)\in \mathcal{J}_1$ such that $b\in L^{q_1}_t L^{p_1}_x$.
    \item[($\mathcal{H}^\sigma_1$)] There exists a constant $K$ such that $\sigma$ is uniformly bounded and nondegenerate, namely
\[
K^{-1} |\xi|^2 \leq |\sigma^\ast (t,x) \xi|^2 \leq K|\xi|^2 \quad
\forall\,\xi\in\R^d,\,(t,x)\in [0,T] \times \R^d.
\]
 	\item[($\mathcal{H}^\sigma_2$)] $\sigma$ is uniformly continuous in space, uniformly in time, with modulus of continuity $h$, in the sense that
\[
|\sigma(t,x)-\sigma(t,y)|\leq h(|x-y|) \quad \forall (t,x,y)\in [0,T]\times \R^{2d}.
\]
	\item[($\mathcal{H}^\sigma_3$)] There exist $(p_2,q_2)\in \mathcal{J}_1$ such that $\nabla \sigma\in L^{q_2}_t L^{p_2}_x$. 
\end{itemize}
\end{assumption}

It is well established that under \cref{main-ass}, the SDE \eqref{eq:sde} is strongly well-posed for any initial $X_0$ independent of $W$, see e.g. \cite{XXZZ2020} and the basic recap in \cref{subsec:preliminaries}.
In this work we are instead interested in deriving \textit{stability estimates}, in the sense of comparing two distinct solutions $X^1$, $X^2$ driven by the same $W$ but with respect to different initial conditions, drifts and diffusions.
The next statement exemplifies some of our main findings in this regard.

\begin{theorem}\label{thm:main-theorem}
Let $(b^1,\sigma^1)$, $(b^2,\sigma^2)$ satisfy \cref{main-ass} for the same parameters $(p_1,q_1)$, $(p_2,q_2)$, $K$, $h$ and denote by $X^1$, $X^2$ the associated solutions to \eqref{eq:sde}.
Then for any $m\in [1,\infty)$, $T>0$ and any $(\tilde p_1,\tilde q_1)\in \mathcal{J}_0$, $(\tilde p_2,\tilde q_2)\in \mathcal{J}_1$ there exists a constant $C$ (depending on all the aforementioned parameters as well as $d$, $\| b^i\|_{L^{q_1}_t L^{p_1}_x}$ and $\| \nabla \sigma^i\|_{L^{q_2}_t L^{p_2}_x}$) such that
\begin{align}\label{eq:main-stab-1}
    \bigg\| \sup_{t\in[0,T]}|X_t^1-X_t^2|\bigg\|_{L^m_\omega}
    \leq C\, \Big[ \| X^1_0-X^2_0\|_{L^m_\omega} + \| b^1-b^2\|_{L^{\tilde q_1}_t L^{{\tilde p_1}}_x}
    + \| \sigma^1-\sigma^2\Vert_{L^{\tilde q_2}_t L^{\tilde p_2}_x} \Big].
\end{align}
If moreover $q_2=\infty$ and $4/q_1+d/p_1<1$, we have
\begin{align}\label{eq:main-stab-2}
    \bigg\| \sup_{t\in[0,T]}|X_t^1-X_t^2|\bigg\|_{L^m_\omega}
    \leq \tilde{C}\, \Big[ \| X^1_0-X^2_0\|_{L^m_\omega} + \| b^1-b^2\|_{L^{q_1}_t W^{-1,p_1}_x}
    + \| \sigma^1-\sigma^2\Vert_{L^{\tilde q_2}_t L^{\tilde p_2}_x}\Big]
\end{align}
where $\tilde{C}$ is another positive constant with similar dependence on the parameters.
\end{theorem}

To the best of our knowledge, the first paper to discuss the importance of quantitative stability estimates is \cite{zhang2016stochastic}.
Therein, in the case of common diffusion $\sigma^1=\sigma^2$, an estimate for $X^1-X^2$ in function of $\| b^1-b^2\|_{L^{q_1}_t L^{p_1}_x}$ is presented in Theorem 1.1-(E); instead another estimate, in the case $b^1=b^2=0$ and $\sigma^1\neq \sigma^2$ is provided in \cite[Lemma 5.3]{zhang2016stochastic}.
The results is however subject to several restrictions, most importantly the assumption $q=p>d+2$ needed to show well-posedness of the corresponding Kolmogorov equation, and the focus on $m=2$.
Later, again for $\sigma^1=\sigma^2$, the stability result was generalized in \cite[Theorem 3.3]{ling2021wong}, for any values $(p,q)\in \mathcal{J}_1$, but again the resulting estimate is only in function of $\| b^1-b^2\|_{L^q_t L^p_x}$ and holds for $m=2$.
A quite general stability estimate, true for any $m\in [1,\infty)$, is presented in \cite[Theorem 3.10]{XieZhang}, where the SDE is also allowed to have a multiplicative Lévy noise; however, the condition  $\nabla b^1\in L_t^{q'}L_x^{p'}$ for suitable $p',q'\in(1,\infty)$ must be assumed to obtain the stability result.
To the best of our knowledge, the first paper to consider stability estimates in negative Sobolev norms (namely in the style of \eqref{eq:main-stab-2}) is \cite{galeati2022distribution}, which considers the more general framework of SDEs driven by fBm; however in the Brownian setting $H=1/2$, it only allows for additive noise $\sigma=\mathbb{I}$ and drifts $b$ of spatial regularity at least $C^\alpha$ for some $\alpha>0$. The same work also shows how such estimates can be used to solve McKean--Vlasov SDEs.
\cref{thm:main-theorem} can be considered a great generalization of all of the aforementioned results, allowing for any value $m\in [1,\infty)$, different integrability coefficients $(p_1,q_1)\neq (p_2,q_2)$ for $b$ and $\sigma$ in \cref{main-ass} as well as parameters $(\tilde q_1, \tilde p_1)$, $(\tilde q_2, \tilde p_2)$ different from them in the estimates \eqref{eq:main-stab-1} and \eqref{eq:main-stab-2}.
Let us finally mention that this work is hugely inspired by the previous work \cite{le2021taming} by one of the authors, where numerical schemes for singular SDEs are considered; although technically it doesn't contain stability estimates, it is clear that the quantity $\varpi_n$ appearing in \cite [Theorem 1.1]{le2021taming} is closely related to the error committed by comparing two SDEs driven respectively by drift $b^n$ and $b$ respectively, in the case of common diffusion $\sigma$.
We refer to \cref{subsec:euler-maruyama} for a deeper discussion on this point.\\[1ex]
The proof of \cref{thm:main-theorem} will be presented in \cref{subsec:main-proof} and is in fact consequence of a more abstract and much more general stability result, see \cref{stab-general}; the same strategy of proof allows for numerous variants, see \cref{rem:interpolation-stability}, as well as similar results for Stratonovich SDEs, cf. \cref{cor:stratonovich}.

In order to motivate the importance of results in the style of \cref{thm:main-theorem}, we will present in \cref{sec:applications} several applications, ranging from existence and uniqueness results for distribution dependent (or McKean--Vlasov) SDEs, criteria for strong compactness of solutions to singular SDEs and Wong--Zakai type theorems.
We will also shortly discuss in \cref{subsec:other} other applications of such results which have already appeared in the literature, although without it being stated explicitly; indeed, we believe one of nice contributions of this work is also to highlight the importance of stability estimates for SDEs in a wide range of problems, which have so far being treated without recognizing their key role in them.

\begin{remark}
For simplicity in this paper we only considered drifts with global bounds in $L^q_t L^p_x$.
Up to technical details, we expect the same strategy to work in many other cases, including:
\begin{itemize}
    \item[a)] $b, \nabla \sigma \in L_t^q\tilde L_x^p$ where $\tilde L_x^p$ is the so called localized $L_x^p$ considered in  \cite{XXZZ2020}; one of the advantages of this case is that bounded $b$ is covered as well.
    \item[b)] The case of $b$ of the form $b=b^1+\ldots+b^n$ with $b^i\in L^{q_i}_t L^{p_i}_x$ with $(q_i,p_i)\in \mathcal{J}_1$.
    \item[c)]  The presence of an additional Lipschitz term, i.e. $b=b^1+b^2$ with $b^1\in L^q_t L^p_x$ and $b^2\in L^\infty_t C^1_x$ as in \cite{zhang2021zvonkin}.
    \item[d)] Coefficients belonging to mixed-normed spaces, i.e. $b\in L^q_t L^{p_1}_{x_1} \ldots L^{p_d}_{x_d}$ with $2/q+\sum_i 1/p_i<1$, as considered in \cite{ling2021strong}. 
\end{itemize}

\end{remark}

\subsection*{Structure of the paper}
We start by recalling in \cref{subsec:preliminaries} some basic facts related to the SDE \eqref{eq:sde} and deriving some key lemmas needed for later; we then pass to prove several stability results, including \cref{thm:main-theorem}, in \cref{subsec:abstract} and \cref{subsec:main-proof}.
We then present several applications in \cref{sec:applications}: well-posedness for some classes of Mckean-Vlasov SDEs in \cref{subsec:MKV}; strong compactness of solutions to singular SDEs \cref{subsec:compactness}; finally, a Wong--Zakai theorem for singular SDEs in \cref{subsec:wong-zakai}.
We conclude with \cref{sec:Wong-Zakai}, containing more details on the proof of the Wong--Zakai theorem under classical regularity conditions, i.e. $b\in C_x^1$  and $\sigma\in C_x^2$.

\subsection*{Notations and conventions}
We always work on a finite time interval $[0,T]$. We write $a\lesssim b$ to mean that there exists a positive constant $C$ such that $a \leq C b$; we use the index $a\lesssim_\lambda b$ to highlight the dependence $C=C(\lambda)$.
For vectors $x,\, y\in \R^d$, we write $x\cdot y:=\sum_{i=1}^n x_i y_i$ and $|x|=\sqrt{x\cdot x}$; for matrices $A,\, B\in \R^{d\times d}$, we use the convention $A:B:=\sum_{1\leq i,j\leq n}A_{ij}B_{ij}$. Goven a real number $z$, we denote by $\lfloor z\rfloor$ and $\{z\}$ respectively its integer and fractional parts.
%

For any $N\in \mathbb{N}$ and $p\in [1,\infty]$, we denote by $L^p(\R^d;\R^N)$ the standard Lebesgue space; when there is no risk of confusion in the parameter $N$, we will simply write $L^p_x$ for short and denote by $\| \cdot\|_{L^p_x}$ the corresponding norm.
Similarly for the Bessel potential spaces $W^{\beta,p}_x=W^{\beta,p}(\R^d;\R^N)$, which are defined for $\beta\in\R$, with corresponding norm
$$\|f\|_{W_x^{\beta, p}}:=\|(\mI-\Delta)^{\beta/2}f\|_{L_x^{ p}};$$
For $\alpha\in [0,+\infty)$, $C_x^\alpha =C^\alpha (\mathbb{R}^d;\R^N)$ stands for the usual H\"older continuous function space, made of continuous bounded functions with with continuous and bounded derivatives up to order $\lfloor \alpha\rfloor\in\mathbb{N}$ and with globally $\{\alpha\}$-H\"older continuous derivatives of order $\lfloor \alpha\rfloor$.

We denote by $C_T=C([0,T];\R^d)$ the path space of continuous functions on $[0,T]$, endowed with the supremum norm $\| \varphi\|_{C_T}=\sup_{t\in [0,T]} |\varphi_t|$.


Given a Banach space $E$ and a parameter $q\in [1,\infty]$, we denote by $L^q_t E = L^q(0,T;E)$ the space of measurable functions $f:[0,T]\to E$ such that
\[
\| f\|_{L^q_t E}:=\Big( \int_0^T \| f_t\|_E^q \dd t\Big)^{\frac{1}{q}} <\infty
\]
with the usual convention of the essential supremum norm in the case $q=\infty$. Similarly, given a probability space $(\Omega,\mathcal{F},\mathbb{P})$ and $m\in [1,\infty)$, we denote by $L^m_\omega E = L^m(\Omega,\mathcal{F},\mathbb{P};E)$ the space of $E$-valued $\mathcal{F}$-measurable random variables $X$ such that
\[
\| X\|_{L^m_\omega E}:= \mathbb{E} \big[ \| X\|_E^m\big]^{\frac{1}{m}}<\infty
\]
where $\mE$ denotes expectation w.r.t. $\mathbb{P}$. The above definitions can be concatenated by choosing at each step a different $E$, so that one can define $L^m_\omega C_T$, $L^q_t L^p_x$ and so on. Whenever $q=p$, we might write for simplicity $L^p_{t,x}$ in place of $L^p_t L^p_x$.
%

\section{Stability estimates}\label{sec:stability}

\subsection{Preliminaries}\label{subsec:preliminaries}

We shortly recall here several facts related to the SDE \eqref{eq:sde} which will be needed in the sequel, including Girsanov theorem, Zvonkin transform and Khasminskii type estimates.
They are by now very classical tools in the study of singular SDEs, see for instance \cite{XXZZ2020} for a nice self-contained presentation (in a slightly different setting).

We start by recalling that any solution $X$ can be regarded after a change of measure as a solution to a driftless SDE, thanks to Girsanov theorem.

\begin{lemma}\label{lem:girsanov}
Let $(b,\sigma)$ satisfy \cref{main-ass} and let $X$ be the unique solution to \eqref{eq:sde}. Then there exists another measure $\mathbb{Q}$, equivalent to $\mathbb{P}$, and a $\mathbb{Q}$-Brownian motion $\tilde{W}$ such that $X$ solves the SDE $\dd X_t = \sigma_t(X_t)\dd \tilde{W}_t$. Furthermore it holds
\begin{align*}
\mathbb{E}_{\mathbb{P}} \bigg[ \Big( \frac{\dd \mathbb{Q}}{\dd \mathbb{P}}\Big)^n \bigg]<\infty\quad \forall\, n\in \mathbb{Z}.
\end{align*}
\end{lemma}

The next classical ingredient we need are suitable estimates for parabolic PDEs of the form \eqref{eq:zvonkin-pde}, strictly related to the so called \textit{Zvonkin transform}.
In the next statement, we are adopting the notations $a=\sigma \sigma^\ast$ and $a:D^2 f = \sum_{i,j} a_{ij} \partial^2_{ij} f$, $b\cdot \nabla f=\sum_{i}b_i\partial_if$.

\begin{lemma}\label{lem:zvonkin}
Let $(b,\sigma)$ satisfy \cref{main-ass}. Then there exists $\lambda_0\geq 1$, depending on $T,d,K,h,p_1,q_1$ and $\| b\|_{L^{q_1}_t L^{p_1}_x}$, such that for all $\lambda\geq \lambda_0$ there is a unique solution $u\in L^{q_1}_t W^{2,{p_1}}_x \cap W^{1,{q_1}}_t L^{p_1}_x$ to the equation
\begin{equation}\label{eq:zvonkin-pde}
\partial_t u + \frac{1}{2}a:D^2 u + b\cdot\nabla u -\lambda u= -b, \quad u\vert_{t=T}=0.
\end{equation}
Furthermore there exists $\eps>0$, depending on $d,p_1,q_1$, such that
\begin{equation}\label{eq:zvonkin-transfor-estim}
\lambda^\eps \| u\|_{L^\infty_t C^1_x} + \| \partial_t u\|_{L^{q_1}_t L^{p_1}_x} + \| u\|_{L^{q_1}_t W^{2,{p_1}}_x} \lesssim \| b\|_{L^{q_1}_t L^{p_1}_x}.
\end{equation}
\end{lemma}

\begin{proof}
The result is a special case of \cite[Theorem 3.2]{XXZZ2020}; indeed, by  time reversal we can reduce ourselves to the case of a forward parabolic equation with $u\vert_{t=0}=0$ as therein.
Let $\eps>0$ be a sufficiently small parameter satisfying $2/q_1+d/p_1>1-4\eps$ and choose $p'\geq p$ large enough so that $d/p'=2\eps$; setting $q'=\infty$ and $\alpha=1+4\eps$, we can apply the result from \cite{XXZZ2020} to get an estimate for $\lambda^\eps \| u\|_{L^\infty_t W^{\alpha,p'}_x}$. The embedding $ W^{\alpha,p'}_x \hookrightarrow C^{\alpha-d/p'}_x \hookrightarrow C^1_x$ finally allows to conclude.
\end{proof}

The next lemma, which allows to control integrals of the form $\int_0^T f_r(X_r)$ for $f$ merely satisfying $L^q_t L^p_x$-integrability, is also classical; its proof comes from a combination of so called \textit{Krylov's estimates} and \textit{Khasminskii's lemma}.

\begin{lemma}\label{lem:khasminskii}
Let $(b,\sigma)$ satisfy \cref{main-ass}, $X$ be a solution to \eqref{eq:sde} and $(\tilde p,\tilde q)\in \mathcal{J}_0$.
Then there exists a constant $C>0$ (depending on $T$, $d$, $K$, $h$, $p_1$, $q_1$, $\| b\|_{L^{q_1}_t L^{p_1}_x}$, $\tilde p$, $\tilde q$) such that
\begin{equation}\label{eq:khasminskii}
\mE\bigg[ \exp\Big(\lambda \int_0^T f_r(X_r)\dd r\Big) \bigg] \leq C \exp\Big(C \lambda^{\tilde q} \| f\|_{L^{\tilde q}_t L^{\tilde p}_x}^{\tilde q}\Big) \quad \forall\, \lambda >0,\ f\in L^{\tilde q}_t L^{\tilde p}_x.
\end{equation}
\end{lemma}

\begin{proof}
By \cite[Theorem 1.1 (A)]{XXZZ2020}, for any $s<t$ it hold 
\[ \mE\bigg[\int_s^t f(X_r)\dd r \Big\vert \mathcal{F}_s\bigg] \lesssim \| f\|_{L^{\tilde q}(s,t;L^{\tilde p}_x)}=: w(s,t)^{1/{\tilde q}},\]
where $w$ is a continuous control; estimate \eqref{eq:khasminskii} then follows from an application of the quantitative version of Khasminskii's lemma from \cite[Lemma 3.5]{le2021taming}, for the choice $\gamma=1/{\tilde q}$.
\end{proof}

Our next ingredient will be an alternative estimate for additive functional $\int_0^t f_r(X_r) \dd r$ in the case where $f$ only enjoys negative regularity; this result is less classical and improves on the one from \cite[Proposition 6.6]{le2021taming}.

\begin{lemma}\label{lemm-no-drift}
Let $b,\,\sigma$ satisfy \cref{main-ass} with $q_2=\infty$, $X$ be a solution to \eqref{eq:sde} and $(\tilde p, \tilde q)\in \mathcal{J}_1$.
Then there exists a constant $C>0$ (depending on $T$, $d$, $K$, $h$, $p_1$, $q_1$, $p_2$, $\| b\|_{L^{q_1}_t L^{p_1}_x}$, $\| \nabla \sigma\|_{L^\infty_t L^{p_2}_x}$, $\tilde p$, $\tilde q$) such that for all smooth function $f$ it holds
\begin{equation}\label{no-drift-distribution}
    \mE\bigg[ \exp\Big(\lambda  \sup_{t\in [0,T]} \Big|\int_0^t f_r(X_r)\dd r\Big| \Big)\bigg] \leq C \exp\Big(C \lambda^{\tilde q} \| f\|_{L^{\tilde q}_t W^{-1,\tilde p}_x}^{\tilde q}\Big) \quad \forall\, \lambda>0.
\end{equation}
By linearity and density, we can extend the definition of $\int_0^\cdot f_r(X_r)\dd r$ as a continuous process to any $f\in L^{\tilde q}_t W^{-1,\tilde p}_x$, in which case estimate \eqref{no-drift-distribution} still holds.
\end{lemma}

\begin{proof}
First observe that, thanks to \cref{lem:girsanov}, we can assume without loss of generality $b\equiv 0$; we can also assume $f$ to be $\R$-valued, otherwise we can argue componentwise, and by homogeneity we can take $\lambda=1$.
Given a smooth $f$, denote by $v$ the solution to the parabolic PDE
\begin{align*}
   \partial_t v+\frac{1}{2} a : D^2 v =f, \quad v\vert_{t=T}=0.
\end{align*}
Since $\dd X_t = \sigma_t(X_t)\dd W_t$, 
by It\^o's formula we get
\begin{align*}
    \int_0^t f_r(X_r)\dd r= v_t(X_t)-v_0(x)-\int_0^t(\sigma^\ast \nabla v)_r(X_r)\cdot \dd W_r.
\end{align*}
On the other hand, thanks to our assumptions and \cite[Theorem A.4]{le2021taming} (together with some standard interpolation estimates), it holds
\begin{equation}\label{eq:pde-distributional-forcing}
\| v\|_{L^{\tilde q}_t W^{1,\tilde p}_x} + \| \partial_t v\|_{L^{\tilde q}_t W^{-1,\tilde p}_x} + \| v\|_{L^\infty_{t,x}} \lesssim \| f\|_{L^{\tilde q}_t W^{-1,\tilde p}_x}
\end{equation}
Therefore we have the estimate
\begin{align*}
\mE\bigg[ \exp\Big( \sup_{t\in [0,T]} \Big|\int_0^t f_r(X_r)\dd r\Big| \Big)\bigg]
\leq e^{2 \| v\|_{L_{t,x}^\infty}} \mE\bigg[ \exp\Big(  \sup_{t\in [0,T]} \Big|\int_0^t (\sigma^\ast \nabla v)_r(X_r)\cdot \dd W_r\Big| \Big)\bigg];
\end{align*}
using Doob's inequality for the submartingale $M_t := \exp\big(\lambda \int_0^t (\sigma^\ast \nabla v)_r(X_r)\cdot \dd W_r\big)$, as well as the basic inequality $e^{|x|} \leq e^x + e^{-x}$, we can find a constant $\kappa>0$ such that
\begin{align*}
\mE\Big[ \sup_{t\in [0,T]} |M_t|\Big] \leq I^+ + I^-,\quad I^\pm := \mE\bigg[ \exp\Big(\pm \kappa  \int_0^T (\sigma^\ast \nabla v)_r(X_r)\cdot \dd W_r \Big)\bigg].
\end{align*}

Let us show how to estimate $I^+$, the other term being treated similarly.
By estimate \eqref{eq:pde-distributional-forcing}, $ \nabla v\in L^{\tilde q}_t L^{\tilde p}_x$, thus $|\sigma^\ast \nabla v|^2\in L^{\tilde q/2}_t L^{\tilde p/2}_x$ where $(\tilde q/2, \tilde p/2)\in \mathcal{J}_0$. In particular  \cref{lem:khasminskii} applies to $f=|\sigma^\ast \nabla v|^2$; but then by Girsanov's theorem and Novikov's condition, this implies that
\[
N_t := \exp\Big(2\kappa \int_0^t (\sigma^\ast \nabla v)_r(X_r)\cdot \dd W_r - 2\kappa^2 \int_0^t |(\sigma^\ast\nabla v)_r(X_r)|^2 \dd t\Big)
\]
is a martingale, which together with H\"older's inequality yields
\begin{align*}
I^+ & \leq \mE[N_t]^{1/2}\, \mE\bigg[\exp\Big(2\kappa^2 \int_0^T |(\sigma^\ast \nabla v)_r(X_r)|^2 \dd r\Big)\bigg]^{1/2}
\lesssim C \exp\Big (C\, \| |\sigma^\ast \nabla v|^2 \|_{L^{\tilde q/2}_t L^{\tilde p/2}_x}^{\tilde q/2}\Big).
\end{align*}
Combining all of the above with the estimate
\[ \| |\sigma^\ast \nabla v|^2 \|_{L^{\tilde q/2}_t L^{\tilde p/2}_x}
\lesssim \| \nabla v\|_{L^{\tilde q}_t L^{\tilde p}_x}
\lesssim \| f\|_{L^{\tilde q}_t W^{-1,\tilde p}_x},\]
where the second passage comes \eqref{eq:pde-distributional-forcing}, and relabelling the constant $C$ to include $\kappa$, we finally obtain the desired \eqref{no-drift-distribution}.
\end{proof}

\begin{remark}
The realisation that suitable diffusion processes can be integrated along distributions goes back at least to the works of Bass and Chen \cite{bass2003brownian}, where the case of measures in a suitable Kato class is considered. Among more recent developments let us mention \cite{flandoli2017multidimensional} and \cite[Theorem 5.1]{zhang2017heat}, where autonomous $f\in W^{-\alpha,p}_x$ are allowed for $\alpha \leq 1/2$, $p>d/(1-\alpha)$. As already mentioned, \cref{lemm-no-drift} is an improvement over \cite{le2021taming}, as it allows $\nu=1$ without further assumptions of $f$ and reaches exponential integrability.

Among other applications of \cref{lemm-no-drift} of independent interest, let us mention that for the choice $f=\div b\in L^{q_1}_t W^{-1,p_1}_x$ it yields uniform $L^m_\omega$ estimates for the Jacobian $J\phi_t(x)=\det D_x\phi_t(x)$ of the flow $\phi$ associated to the SDE \eqref{eq:sde}, which is formally given by the formula
\[
J\phi_t (x) = \exp\Big(\int_0^t \div b_r (X^x_r) \dd r\Big).
\]
Compare this fact with \cite[Section 3]{flandoli2010well}, the additional assumption $\div b\in L^p$, $p>2$ is imposed in order to analyse such term and derive regularity estimates.
\end{remark}

\begin{remark}\label{rem:interpolation-khasminskii}
Interpolating between $L^{\tilde q_0}_t L^{\tilde p_0}_x$ and $L^{\tilde q_1}_t W^{-1,\tilde p_1}_x$ for $(\tilde q_0,\tilde p_0)\in\mathcal{J}_0$, $(\tilde q_1,\tilde p_1)\in\mathcal{J}_1$ and applying Lemmas \ref{lem:khasminskii}-\ref{lemm-no-drift}, one can in fact prove a more general statement: for any $\beta\in [0,1]$, $(\tilde q_\beta,\tilde p_\beta)\in \mathcal{J}_\beta$ and any $f\in L^{\tilde q_\beta}_t W^{-\beta,\tilde p_\beta}_x$, the process $\int_0^\cdot f_r(X_r)\dd r$ is well defined and
\begin{equation*}
\mE\bigg[ \exp\Big(\lambda  \sup_{t\in [0,T]} \Big|\int_0^t f_r(X_r)\dd r\Big| \Big)\bigg] \leq C \exp\Big(C \lambda^{\tilde q_\beta} \| f\|_{L^{\tilde q_\beta}_t W^{-\beta,\tilde p_\beta}_x}^{\tilde q_\beta}\Big) \quad \forall\, \lambda>0.
\end{equation*}
\end{remark}

\begin{remark}
Estimates \eqref{eq:khasminskii}-\eqref{no-drift-distribution} immediately imply moment bounds of the form
\begin{align}\label{est:khasminskii}
  \mE\bigg[ \Big(\int_0^T f_r(X_r)\dd r\Big)^m \bigg] \lesssim  \| f\|_{L^{\tilde q}_t L^{\tilde p}_x}^m, \quad
  \mE\bigg[  \sup_{t\in [0,T]} \bigg|\int_0^t f_r(X_r)\dd r\Big| ^m\bigg] \lesssim  \| f\|_{L^{\tilde q}_t W^{-1,\tilde p}_x}^m
\end{align}
for all $m\in [1,\infty)$. For instance, to deduce the first estimate in \eqref{est:khasminskii}, it suffices to  apply \eqref{eq:khasminskii} for the choice $\lambda =1$, $\tilde{f} = f/ \| f\|_{L^{\tilde q}_t L^{\tilde p}_x}$, so that
\begin{align*}
\| f\|_{L^{\tilde q}_t L^{\tilde p}_x}^{-m}\, \mE\bigg[ \Big(\int_0^T f_r(X_r)\dd r\Big)^m \bigg]
\lesssim_m \mE\bigg[ \exp\Big( \int_0^T \tilde f_r(X_r)\dd r \Big)\bigg] \lesssim 1.
\end{align*}
\end{remark}

As a final ingredient, we need to recall an inequality involving maximal functions; the standard version, sometimes referred to as Hajlasz inequality, is classical \cite{stein1970singular}, but the one-sided one presented here seems to be of much recent derivation, cf. \cite{caravenna2021}.

\begin{lemma}\label{lem:maximal-one-sided}
Let $f$ be a measurable function such that $\nabla f=\psi^1+\psi^2$, where $\psi^1\in L^{\tilde q_i}_t L^{\tilde p_i}_x$ for some parameters $(\tilde q_1,\tilde p_1),\, (\tilde q_2,\tilde p_2)$  satisfying $\tilde q_i\in [1,\infty)$ and $\tilde p_i\in (d,\infty)$; then there exists a negligible Lebesgue set $N\subset [0,T]\times \R^d$ and another measurable function $g$ such that
\begin{equation}\label{eq:maximal-one-sided}
    |f_t(x)-f_t(y)| \leq [g^1_t(x) + g^2_t(x)] |x-y| \quad \forall\,y\in \R^d,\ \forall\, (t,x)\in [0,T]\times\R^d\setminus N;
\end{equation}
moreover there exists a constant $C=C(d,\tilde p_1,\tilde p_2)$ such that $g$ can be chosen of the form $g=g^1+g^2$ satisfying $\| g^i\|_{L^{\tilde q^i}_t L^{\tilde p^i}_x} \leq C \| \psi^i\|_{L^{\tilde q^i}_t L^{\tilde p^i}_x}$ for $i=1,2$.
\end{lemma}

\begin{proof}
In the case of autonomous functions $f$, the statement is a simple consequence of \cite[Lemma 5.1]{caravenna2021}; more precisely, by fixing another $p'\in (d,\tilde p_1 \wedge \tilde p_2)$ and applying the result therein, inequality \eqref{eq:maximal-one-sided} holds for $g(x)\sim [ \mathcal{M}(|\nabla f|^{p'})(x)]^{1/p'}$, where $\mathcal{M}$ denotes the Hardy--Littlewood maximal function operator.
By standard properties of $\mathcal{M}$, $g\lesssim g^1+g^2$ for $g^i(x):=\mathcal{M}(|\psi^i|^{p'})(x)]^{1/p'}$; by construction, $|\psi^i|^{p'} \in L^{\tilde p_i/p'}_x$ with $\tilde p_i /p'>1$, therefore again by properties of $\mathcal{M}$ it holds $\|\mathcal{M}(|\psi^i|^{p'})\|_{L^{p/p'}} \lesssim  \| \psi^i\|_{L^{\tilde p_i}_x}^{p'}$, from which we can deduce that $\| g^i\|_{L^{\tilde p_i}_x} \lesssim \| \psi^i\|_{L^{\tilde p_i}_x}$.

In the case of time-dependent functions $f$, one can either: i) observe that for a.e. $t\in [0,T]$, $\nabla f_t\in L^{\tilde p_1}_x + L^{\tilde p_2}_x$ and apply the result at such fixed $t$; or ii) consider a sequence $f^n$ of smooth functions such that $f^n\to f$ and they admits decompositions $\nabla f^n=\psi^{n,1}+\psi^{n,2}$ with $\psi^{n,i}\to \psi^i$ in $L^{\tilde q_i}_t L^{\tilde p_i}_x$; then to any $f^n$ there is an associated $g^n=g^{n,1}+g^{n,2}= \mathcal{M}(|\psi^{n,1}|^{p'})]^{1/p'} + \mathcal{M}(|\psi^{n,2}|^{p'})]^{1/p'}$ such that \eqref{eq:maximal-one-sided} holds and clearly $g^{n,i} \to g^i$ in $L^{\tilde q_i}_t L^{\tilde p^i}_x$, so that inequality \eqref{eq:maximal-one-sided} still holds after the limit (up to a Lebesgue negligible set). 
\end{proof}

\subsection{A general abstract estimate}\label{subsec:abstract}

We present here a fundamental estimate which is at the heart of the proof of Theorem \ref{thm:main-theorem}, as well as of several applications.
It focuses on comparing the solution $X$ to the SDE \eqref{eq:sde} associated to $(X_0, b,\sigma)$ w.r.t. to any semimartingale perturbation of it; more precisely, we will consider a process $Y$ solving
\begin{equation}\label{eq:perturbed-sde}
Y_t = Y_0 + \int_0^t [b_s(Y_s) + R^1_s] \dd s + \int_0^t [\sigma_s(Y_s) + R^2_s] \dd W_s
\end{equation}
where $R^1$, $R^2$ are predictable processes, taking values respectively in $\R^d$ and $\R^{d\times d}$.
In the following, we will enforce the following conditions on $Y$.

\begin{assumption}\label{ass:perturbation}
The following hold for any $T>0$:
\begin{itemize}
\item[i)] $\int_0^T [|b_s(Y_s)| + |R^1_s|+|\sigma_s(Y_s)|^2+|R^2_s|^2]\dd s <\infty$ $\mathbb{P}$-a.s.
\item[ii)] For any $f\in L^{\tilde q}_t L^{\tilde p}_x$ with $(\tilde{q},\tilde{p})\in \mathcal{J}_0$, it holds $\int_0^T |f_r(Y_r)| \dd r<\infty$ $\mathbb{P}$-a.s. Moreover for any sequence of smooth, bounded functions $f^\eps$ such that $f^\eps\to f$ in $L^{\tilde q}_t L^{\tilde p}_x$ as $\eps\to 0$, it holds
\begin{equation*}
\lim_{\eps\to 0} \int_0^T |f^\eps(Y_s)-f(Y_s)| \dd s = 0 \quad \mathbb{P}\text{-a.s.}
\end{equation*}
\item[iii)] The processes $Y$, $R^1$, $R^2$ are adapted to the filtration generated by $(Y_0,W,\tilde W)$, where $\tilde W$ is another $\R^m$-valued process, for some $m\in\mathbb{N}$, and  $Y_0$, $W$, $\tilde W$ are independent.
\end{itemize}
\end{assumption}

\begin{remark}
While Condition i) is natural in order to give meaning to \eqref{eq:perturbed-sde}, the others are less classical.
Point ii) will be needed in order to make sense of an It\^o-type formula for functions with only integrable second order derivatives.
Instead iii) simplifies the proof of \cref{stab-general} below by allowing to disintegrate w.r.t. the initial conditions $(X_0,Y_0)$; depending on the specific structure of the processes $R^1$, $R^2$ it might not be even needed. In the simplest cases, one can just take $\tilde W\equiv 0$, but there are also applications of interest where a nontrivial choice of $\tilde{W}$ must be considered,  as we will see in \cref{sec:applications}.
\end{remark}

Before presenting the statement, we need a few preparations. Given $(b,\sigma)$ satisfying \cref{main-ass}, by \cref{lem:zvonkin} we can choose $\lambda$ large enough so that the solution $u$ to \eqref{eq:zvonkin-pde} has Lipschitz constant smaller than $1/2$ for all $t\in [0,T]$; from now on, we will always work with such a fixed $\lambda$ and denote by $u$ the associated solution, so that the map $\phi_t(x):=x+u_t(x)$ is $2$-Lipschitz with $2$-Lipschitz inverse for all $t\in [0,T]$.
For such $u$, we define an associated process
\begin{equation}\label{eq:process-V}
V_t := \int_0^t \Big(\frac{1}{2}\big[ (R^2+\sigma)(R^2+\sigma)^\ast - \sigma \sigma^\ast\big] :D^2 u +R^1\cdot(I+\nabla u) \Big)_r(Y_r) \dd r + \int_0^t [R^2(I + \nabla u )]_r(Y_r)\dd W_r.
\end{equation}
which is well defined thanks to \cref{ass:perturbation}.

The proof of the next result shares some clear analogies with those of stochastic Gr\"onwall lemmas, cf. \cite[Lemma 3.8]{le2021taming}, as well as \cite{geiss2022concave} for a general overview; however, in order not to impose any restriction on the $m$-th moment we want to estimate, we will exploit crucially the structure of the SDE which allows us to ``recenter'' it at each step, a property which doesn't hold in the more general setting of those results.

\begin{proposition}\label{stab-general}
Let $X$ be a solution to \eqref{eq:sde} associated to $(X_0, b, \sigma)$ satisfying \cref{main-ass} and let $Y$ be a solution to \eqref{eq:perturbed-sde} satisfying \cref{ass:perturbation}. Then for any $m\in [1,\infty)$, any $\gamma>1$ and any $T>0$ there exists a constant $C$ (depending on $m,\,\gamma,\,T$, $d$, $(q_i,p_i)$, $K$, $h$, $\| b\|_{L^{q_1}_t L^{p_1}_x}$, $\| \nabla \sigma\|_{L^{q_2}_t L^{p_2}_x}$), such that
\begin{equation}\label{eq:stab-general}
\Big\| \sup_{t\in [0,T]} |X_t-Y_t| \Big\|_{L_\omega^m} \leq C \Big(\| X_0-Y_0\|_{L_\omega^m} + \Big\| \sup_{t\in [0,T]} | V_t| \Big\|_{L_\omega^{m\gamma}} \Big) 
\end{equation}
where the process $V$ is defined as in \eqref{eq:process-V}.
\end{proposition}

\begin{proof}
It suffices to consider the case where the initial data $X_0$ and $Y_0$ are deterministic; indeed in the general case, as they are random variables independent of the driving $(W,\tilde W)$, we can first condition on the knowledge of $(X_0,Y_0)$, obtain the estimate, and then take expectation  again to conclude. From now on we will fix $m\in [1,\infty)$. 

We start by applying Zvonkin transform to both $X$ and $Y$; by classical computations, setting $\tilde{\sigma}:= \sigma(I+\nabla u)$ it holds
\[
\dd \big( \phi_t(X_t)\big) = \lambda u_t(X_t) \dd t + \tilde{\sigma}_t(X_t) \dd W_t;
\]
the computation for $Y$ is slightly more involved, but it is not difficult to check that, thanks to our definition of $V$, in the end one arrives at
\begin{align*}
\dd \big( \phi_t(Y_t)\big) = \lambda u_t(Y_t) \dd t + \tilde{\sigma}_t(Y_t) \dd W_t + \dd V_t.
\end{align*}
Indeed, the computation is  rigorously justified thanks to \cref{ass:perturbation}, as we can first apply It\^o formula to $\phi^\eps$, associated to some smooth approximation $b^\eps$ of $b$, and then pass to the limit as $\eps\to 0$.
As a consequence, if we integrate both equations over any interval $[s,t]$ and take the modulus, we find the estimate
\begin{align*}
|\phi_t(X_t)-\phi_t(Y_t)|\leq &
|\phi_s(X_s)-\phi_s(Y_s)| 
+ \lambda \int_s^t |u_r(X_r)-u_r(Y_r)| \dd r \\&
+ \bigg|\int_s^t [\tilde{\sigma}^1_r(X_r)-\tilde{\sigma}^1_r(Y_r)] \dd W_r\bigg| + \bar{V}
\end{align*}
where we set $\bar{V}:=2 \sup_{t\in [0,T]} V_t$; using the properties of $\phi,\,u$, we arrive at
\begin{equation}\label{eq:estim-1}
\sup_{t\in [s,s+h]} |X_t-Y_t| \lesssim |X_s-Y_s| + \lambda\, h \sup_{t\in [s,s+h]} |X_t-Y_t| + \sup_{t\in [s,s+h]} \bigg|\int_s^t [\tilde{\sigma}^1_r(X_r)-\tilde{\sigma}^1_r(Y_r)] \dd W_r\bigg| + \bar{V}.
\end{equation}
Observe that this is a pathwise inequality that holds for any choice of $s,h$, which we may and will take to be random as well.

Moreover, since $\nabla\sigma \in L^{q_2}_t L^{p_2}_x$ and $u\in L^\infty_t C^1_x \cap L^{q_1}_t W^{2,{p_1}}_x$, the new diffusion term satisfies $\nabla\tilde{\sigma}\in L^{q_1}_t L^{p_1}_x + L^{q_2}_t L^{p_2}_x $ as well; therefore we can apply \cref{lem:maximal-one-sided} to find functions $g^i\in L^{q_i}_t L^{p_i}_x$, $i=1,2$, for which \eqref{eq:maximal-one-sided} holds for $f=\tilde{\sigma}$.
Let us define the process
\begin{equation*}
A_t:= \int_0^t [\lambda^2 T + |g^1_r|^2(X_r) + |g^2_r|^2(X_r)] \dd r;
\end{equation*}
for a constant $\tilde K>0$ to be chosen later; we also set an increasing sequence of stopping times by $\tau^0=0$ and
\begin{equation*}
\tau^{n+1}=\inf\big\{t\geq \tau^n: t\leq T,\ A_t-A_{\tau^n} \geq  (4\tilde K)^{-\frac{2}{m}}\big\},
\end{equation*}
%
with the convention that $\tau_{n+1}=T$ if there is no $t\in [\tau_n,T]$ such that $A_t-A_{\tau^n}=(4\tilde K)^{-2/m}$.
Now let us take $s=\tau^n$ in \eqref{eq:estim-1} and elevate both sides to the $m$-th power, so to obtain
\begin{align*}
\sup_{t\in [\tau^n,\tau^n+h]} |Z_t|^m \lesssim_m \big[ |Z_{\tau^n}|^m + \lambda^m \, h^m \sup_{t\in [\tau^n,\tau^n+h]} |Z_t|^m + \sup_{t\in [0,h]} |M^n_t|^m + \bar{V}^m\big]
\end{align*}
where we set $Z:=X-Y$ and
\begin{equation*}
M^n_t:=\int_{\tau^n}^{\tau^n+t} [\tilde{\sigma}_r(X_r)-\tilde{\sigma}_r(Y_r)] \dd W_r.
\end{equation*}
By Doob's optimal stopping theorem $M^n_t$ is a continuous martingale w.r.t. to the filtration $\mathcal{G}_t=\mathcal{F}_{\tau^n+t}$; by the pathwise BDG inequality from \cite[Theorem 3]{siorpaes2018}, there exists another continuous $\mathcal{G}_t$-martingale $\tilde{M}^n$ such that
\[ \sup_{t\in [0,h]} |M^n_t|^m\leq [M^n]_h^{m/2} + \tilde{M}^n_h \]
%
where $[M^n]$ denotes the quadratic variation of $M^n$; by \cref{lem:maximal-one-sided} it holds
\[
[M^n]_t = \int_{\tau^n}^{\tau^n+t} |\tilde{\sigma}_r(X_r)-\tilde{\sigma}_r(Y_r)|^2 \dd r
\leq \int_{\tau^n}^{\tau^n+t} \big[ |g^1_r|^2(X_r) + |g^2_r|^2(X_r)\big] |Z_r|^2 \dd r.
\]
Combining everything and taking $h=\tau_{n+1}-\tau_n$, we arrive at an inequality of the form
\begin{align*}
\sup_{t\in [\tau^n,\tau^{n+1}]} |Z_t|^m
& \leq C \Big[ |Z_{\tau^n}|^m + \tilde M^n_{\tau^{n+1}-\tau^n} + |\bar{V}|^m\Big]\\
&\quad + \sup_{t\in [\tau^n,\tau^{n+1}]} |Z_t|^m \Big[\lambda^m \, T^{m/2} (\tau^{n+1}-\tau^n)^{m/2}
+\Big( \int_{\tau^n}^{\tau^{n+1}} (|g_1|^2+|g_2|^2)(X_r) \dd r\Big)^{m/2} \Big]\\
& \leq C \Big[ |Z_{\tau^n}|^m + \tilde M^n_{\tau^{n+1}-\tau^n} + |\bar{V}|^m\Big] + 2 C\, (A_{\tau^{n+1}}-A_{\tau^n})^{m/2} \sup_{t\in [\tau^n,\tau^{n+1}]} |Z_t|^m
\end{align*}
%
Taking $C=\tilde K$ as the desired constant in the definition of $\tau^n$, we find
\begin{equation*}
\sup_{t\in [\tau^n,\tau^{n+1}]} |Z_t|^m \leq 2\tilde K \Big[ |Z_{\tau^n}|^m + \tilde M^n_{\tau^{n+1}-\tau^n} + |\bar{V}|^m\Big];
\end{equation*}
using the fact that $\tilde{M}^n$ is a $\mathcal{G}_t$-martingale, taking expectation we are left with
\begin{equation*}
\mE\Big[\sup_{t\in [\tau^n,\tau^{n+1}]} |Z_t|^m \bigg ]
\leq 2\tilde K\, \mE\Big[ \sup_{t\in [\tau^{n-1},\tau^{n}]} |Z_t|^m \Big] + 2 \tilde K\, \mE[|\bar{V}|^m],
\end{equation*}
which iteratively yields (for an appropriately chosen constant $\kappa>0$ in function of $\tilde K$)
\begin{equation}\label{eq:estim-2}
\mE\bigg[ e^{-\kappa n} \sup_{t\in [\tau^{n-1},\tau^n]} |Z_t|^m \bigg]\leq |x_0-y_0|^m+\mE[|\bar{V}|^m].
\end{equation}
Now set $\Delta:= (4 \tilde K)^{-2/m}$, $C:= 2\kappa/ \Delta$, then it holds
\begin{align*}
\mE\Big[ e^{-C A_T} \sup_{t\in [0,T]} |Z_t|^m \Big]
& = \sum_{n\in \mathbb{N}} \mE\Big[ e^{-C A_T} \sup_{t\in [0,T]} |Z_t|^m \mathbbm{1}_{A_T\in [\Delta n, \Delta (n+1)]} \Big]\\
& \leq \sum_{n\in \mathbb{N}} \mE \Big[ e^{-2\kappa n} \sup_{j=1,\ldots,n+1} \sup_{t\in [\tau_{j-1},\tau_j]} |Z_t|^m \mathbbm{1}_{A_T\in [\Delta n, \Delta (n+1)]}\Big]\\
& \leq \sum_{n\in\mathbb{N}} e^{-\kappa n} \sum_{j=1}^{n+1} \mE\Big[ e^{-\kappa n} \sup_{t\in [\tau_{j-1},\tau_j]} |Z_t|^m\Big]\\
& \lesssim  |x_0-y_0|^m+\mE[|\bar{V}|^m]
\end{align*}
where in the last passage we applied \eqref{eq:estim-2}. Observe that since $g^i\in L^{q_i}_t L^{p_i}_x$, $|g^i|^2\in L^{{q_i}/2}_t L^{{p_i}/2}_x$ with $({q_i}/2,{p_i}/2)\in \mathcal{J}_0$, thus by \cref{lem:khasminskii} the process $A$ satisfies $\mE[\exp(\lambda A)]<\infty$ for all $\lambda\in \R$; therefore for any $\delta\in (0,1)$ it holds
\begin{align*}
\mE\Big[\sup_{t\in [0,T]} |Z_t|^{\delta m} \Big]
& \leq \mE\Big[e^{-C A_T} \sup_{t\in [0,T]} |Z_t|^{m} \Big]^{\delta}\, \mE\Big[e^{C \frac{\delta}{1-\delta} A_T}\Big]^{1-\delta}\\
& \lesssim_\delta |x_0-y_0|^{\delta m} +\mE[|\bar{V}|^m]^\delta.
\end{align*}
Taking the $(\delta m)^{-1}$-power on both sides and relabelling $\tilde{m}=m\delta$, $\gamma \tilde{m} =m$, overall yields the desired result.
\end{proof}

\subsection{Comparison of SDEs with different coefficients}\label{subsec:main-proof}
With \cref{stab-general} at hand, we are now ready to give the 

\begin{proof}[Proof of \cref{thm:main-theorem}]
Let us set $X^1=X$, $X^2=Y$, $b^1=b$, $\sigma^1=\sigma$, $R^1_t=(b^2-b^1)_t(X^2_t)$ and $R^2_t=(\sigma^2-\sigma^1)_t(X^2_t)$ in the setting of \cref{subsec:abstract}; it is then clear by \cref{subsec:preliminaries} that \cref{ass:perturbation} holds and so by \cref{stab-general} our task is reduced to provide estimates for $\| \bar{V}\|_{L^{m\gamma}}$;
The process $V$, as defined in \eqref{eq:process-V}, in this case can be decomposed as $V=I^1+I^2+I^3$ for
\begin{align*}
    &I^1_t:= \frac{1}{2} \int_0^t [(a^2-a^1):D^2 u ]_r(X^2_r) \dd r, \quad
    I^2_t:= \int_0^t [(\sigma^2-\sigma^1)(I+\nabla u)]_r(X^2_r)\dd W_r,\\
    & I^3_t:=\int_0^t[(b^2-b^1)\cdot(I+\nabla u)]_r(X^2_r) \dd r
\end{align*}
where $a^i=\sigma^i (\sigma^i)^\ast$ for $i=1,2$; we treat each term separately. From now on for simplicity we will write $m$ in place of $m\gamma$, as it doesn't affect the computation.

Recall that, under \cref{main-ass}, by Lemma \cref{lem:zvonkin} it holds $u\in L^{q_1}_t W^{2,{p_1}}_x \cap L^\infty_t C^1_x$; moreover $(p_1,q_1),\, (\tilde p_2,\tilde q_2)\in \mathcal{J}_1$, thus if we define $1/p'=1/p_1+1/\tilde p_2$ and $1/q'=1/q_1+1/\tilde q_2$, it holds $(p',q')\in \mathcal{J}_0$. By H\"older's inequality and estimate \eqref{est:khasminskii}, we can then control $I^1$ by
\begin{equation}\label{est-I1}\begin{split}
    \Big\| \sup_{t\in [0,T]} | I^1_t | \Big\|_{L_\omega^{m}}
    &\lesssim \|  (a^2-a^1):D^2 u \Vert_{L^{ q'}_t L^{ p'}_x}\\
    & \lesssim \|  a^2-a^1\Vert_{L^{ \tilde q_2}_t L^{\tilde p_2}_x} \Vert D^2 u \Vert_{L^{q_1}_t L^{p_1}_x}
    \lesssim \|  \sigma^2-\sigma^1\Vert_{L^{ \tilde q_2}_t L^{\tilde p_2}_x}
\end{split}\end{equation}
where in the final passage we also exploited the bound $\|\sigma^i\|_{L^\infty_{t,x}} \leq K$.
A similar idea applies to $I^2$, combined with BDG's inequality:
\begin{equation}\label{est-I2}\begin{split}
    \Big\| \sup_{t\in [0,T]} | I^2_t | \Big\|_{L_\omega^{m}}
    & \lesssim \Big\Vert  \int_0^T |(\sigma^2_r-\sigma^1_r)(I+\nabla u_r)|^2(X^2_r)\dd r\Big\Vert_{L_\omega^{m/2}}^{1/2}\\
    &\lesssim \big\Vert |\sigma^2-\sigma^1|^2 \big\Vert_{L^{ \tilde q_2/2}_t L^{\tilde p_2/2}_x}^{1/2}\Vert I+\nabla u\Vert_{L^{\infty}_t L^{\infty}_x}
    \lesssim \Vert \sigma^1-\sigma^2\Vert_{L^{ \tilde q_2}_t L^{\tilde p_2}_x}.
\end{split}\end{equation}
where this time we applied \eqref{est:khasminskii} for $(\tilde p_2/2, \tilde q_2/2)\in\mathcal{J}_0$.
Similarly, since $(\tilde{p}_1,\tilde{q}_1)\in\mathcal{J}_0$, it holds
\begin{equation}\label{est-I3-1}\begin{split}
     \Big\| \sup_{t\in [0,T]} | I^3_t | \Big\|_{L_\omega^{m}}
     &\lesssim \|  (b^2-b^1)\cdot(I+\nabla u) \Vert_{L^{ \tilde p_1}_t L^{ \tilde q_1}_x}\\
     & \lesssim \|  b^2-b^1 \Vert_{L^{ \tilde p_1}_t L^{ \tilde q_1}_x} \Vert (I+\nabla u) \Vert_{L^{ \infty}_t L^{ \infty}_x}
    \lesssim \Vert  b^2-b^1\Vert_{L^{\tilde p_1}_t L^{ \tilde q_1}_x}.
\end{split}\end{equation}
In the case $q_2=\infty$ and $4/q_1+d/p_1<1$, we can give an alternative estimate for $I^3$; to this end, let us recall that by the product rules for distributions (see e.g. \cite[Lemma A.2 (ii)]{le2021taming}), for $p\geq d \vee 2 $ it holds
\begin{equation}\label{eq:product-distributions}
\| f\, g\|_{W^{-1,p}_x} \lesssim \| f\|_{W^{-1,p}_x} \| g\|_{W^{1,p}_x}\quad \forall\, f\in W^{-1,p}_x,\, g\in W^{1,p}_x.
\end{equation}
By H\"older's inequality and the second estimate in \eqref{est:khasminskii} (for $\tilde{q}=q_1/2,\, \tilde{p}=p_1$, which under our assumptions satisfy $(\tilde q,\tilde p)\in\mathcal{J}_1$), we then find
\begin{equation}\label{est-I3-2}
\Big\| \sup_{t\in [0,T]} | I^3_t | \Big\|_{L_\omega^{m}}
\lesssim \|  (b^2-b^1)\cdot(I+\nabla u) \Vert_{L^{ q_1/2}_t W^{-1,p_1}_x}
\lesssim \|  b^2-b^1\Vert_{L^{q_1}_t W^{-1,p_1}_x} \Vert u \Vert_{L^{q_1}_t W^{2,p_1}_x}.
\end{equation}
Overall, estimates \eqref{est-I1}, \eqref{est-I2}, combined with \eqref{est-I3-1} (respectively \eqref{est-I3-2}) and \cref{stab-general} yield \eqref{eq:main-stab-1} (respectively \eqref{eq:main-stab-2}).
\end{proof}

We can also obtain stability estimates for Stratonovich SDEs.

\begin{corollary}\label{cor:stratonovich}
Let \cref{main-ass} hold and consider $X^i$ solutions to the Stratonovich SDEs
\[
\dd X^i_t = b^i_t(X^i_t)\dd t + \sigma^i_t(X^i_t)\circ\! \dd  W_t, \quad X^i\big\vert_{t=0}=X^i_0.
\]
 Then for any $m\in [1,\infty)$ and any $(\tilde p_1,\tilde q_1), (\tilde p_2,\tilde q_2)\in \mathcal{J}_0$ and $(\tilde p_3,\tilde q_3)\in \mathcal{J}_{1}$, it holds
\begin{equation}\label{eq:stratonovich-stab-1}\begin{split}
    \bigg\| \sup_{t\in[0,T]}|X_t^1-X_t^2|\bigg\|_{L^m_\omega}
    \lesssim &\| X^1_0-X^2_0\|_{L^m_\omega} + \| b^1-b^2\|_{L^{\tilde q_1}_t L^{{\tilde p_1}}_x}\\
    & + \| \nabla(\sigma^1-\sigma^2)\|_{L^{\tilde q_2}_t L^{\tilde p_2}_x}
    + \| \sigma^1-\sigma^2\Vert_{L^{\tilde q_3}_t L^{\tilde p_3}_x}.
\end{split}\end{equation}
If additionally $q_2=\infty$ and $4/q_1+d/p_1<1$, we have
\begin{align}\label{eq:stratonovich-stab-2}
    \bigg\| \sup_{t\in[0,T]}|X_t^1-X_t^2|\bigg\|_{L^m_\omega}
    \lesssim \| X^1_0-X^2_0\|_{L^m_\omega} + \| b^1-b^2\|_{L^{q_1}_t W^{-1,p_1}_x}
    + \| \sigma^1-\sigma^2\Vert_{L^{\tilde q_3}_t L^{\tilde p_3}_x}.
\end{align}
\end{corollary}

\begin{proof}
The proof is very similar to that of \cref{thm:main-theorem}, so we mostly sketch it. We can rewrite the Stratonovich SDEs into the corresponding It\^o ones, with corresponding new drift terms $\tilde{b}^i$ in place of $b^i$ given by
\[
\tilde{b}^i_t:= b^i_t + \frac{1}{2} (\sigma^i\cdot\nabla)\sigma^i, \quad
\text{where }\ [(\sigma\cdot\nabla)\sigma]_j = \sum_{k,l} \partial_l \sigma_{jk} \sigma_{lk}.
\]
Estimates \eqref{eq:stratonovich-stab-1}-\eqref{eq:stratonovich-stab-2} then follow respectively from \eqref{eq:main-stab-1}-\eqref{eq:main-stab-2}, once we estimate the difference $\tilde{b}^1-\tilde{b}^2$ in appropriate norms.
In order to obtain \eqref{eq:stratonovich-stab-1}, we can use
\[
(\sigma^1\cdot \nabla) \sigma^1-(\sigma^2\cdot \nabla) \sigma^2 = (\sigma^1\cdot \nabla)(\sigma^1-\sigma^2) + [(\sigma^1-\sigma^2)\cdot\nabla] \sigma^2
\]
and then
\begin{align*}
& \| (\sigma^1\cdot \nabla)(\sigma^1-\sigma^2)\|_{L^{\tilde q_2}_t L^{\tilde p_2}_x}
\leq \| \sigma^1\|_{L^\infty_{t,x}} \| \nabla \sigma^1-\nabla\sigma^2\|_{L^{\tilde q_2}_t L^{\tilde p_2}_x}\\
&\| [(\sigma^1-\sigma^2)\cdot\nabla] \sigma^2\|_{L^{q'}_t L^{p'}_t}
\leq \|\sigma^1-\sigma^2\|_{L^{\tilde q_3}_t L^{\tilde p_3}_x} \, \| \nabla\sigma^2\|_{L^{q_2}_t L^{p_2}_x}
\end{align*}
where we set $1/p'=1/\tilde{p}_3+1/p_2$, $1/q'=1/\tilde{q}_3+1/q_2$ so that $(p',q')\in\mathcal{J}_0$.
Instead for obtaining \eqref{eq:stratonovich-stab-2}, in the case $q_2=\infty$ we can estimate the first term differently. To this end, let us recall that by \cite[Lemma A.2 (iii)]{le2021taming}), whenever $p>d$ and $1/\bar{p}+1/p<1$, it holds
\begin{equation}\label{eq:product-distributions-2}
\| f\, g\|_{W^{-1,\bar{p}}_x} \lesssim (\| g\|_{L^\infty_x}+ \| \nabla g\|_{L^p_x}) \| f\|_{W^{-1,\bar{p}}_x}\quad \forall\, f\in W^{-1,p}_x,\, g\in W^{1,p}_x.
\end{equation}
Applying estimate \eqref{eq:product-distributions-2} for the choice $p=p_2$ and $\bar{p}=\tilde{p}_3$, we obtain
\begin{align*}
\| \sigma^1\cdot\nabla (\sigma^1-\sigma^2)\|_{L^{\tilde q_3}_t W^{-1,\tilde p_3}_x}
& \lesssim \big[ \| \sigma^1\|_{L^\infty_{t,x}} + \| \nabla \sigma^1\|_{L^\infty_t L^p_x}\big] \| \nabla (\sigma^1-\sigma^2)\|_{L^{\tilde q_3}_t W^{-1,\tilde p_3}_x}
\\&\lesssim \| \sigma^1-\sigma^2\|_{L^{\tilde q_3}_t L^{\tilde p_3}_x}.
\end{align*}
Combined with \cref{stab-general} and estimates \eqref{est-I1}, \eqref{est-I2}, \eqref{est-I3-1} and \eqref{est-I3-2}, this yields the conclusion.
\end{proof}

\begin{remark}\label{rem:interpolation-stability}
\cref{thm:main-theorem} and \cref{cor:stratonovich} admit several other variants, in the sense of employing other quantities of the form $\| b^1-b^2\|_{L^{\tilde q_1}_t W^{-\beta,\tilde p_1}_x}$ in  the r.h.s. of \eqref{eq:main-stab-2} and \eqref{eq:stratonovich-stab-2}. Indeed going through the proof, the only relevant difference comes from estimating the term $I^3$ related to $\int_0^\cdot [(b^2-b^2)\cdot (I+\nabla u)]_r(X^2_r) \dd r$, which can be accomplished as in \cref{rem:interpolation-khasminskii} under the assumption $q_2=\infty$. There are two different regimes, related to the parameters $(\beta,q_1,p_1)$:
\begin{enumerate}
\item $\beta < 1-2/q_1$. In this case, since $(p_1,q_1)\in\mathcal{J}_1$, we can assume $\beta$ to satisfy $\beta>d/p_1$ as well (this proves a stronger estimate, as for $\tilde \beta\leq \beta$ one can use $W^{-\tilde\beta,p}_x\hookrightarrow W^{-\beta,p}_x$). By \cref{lem:zvonkin} and standard interpolation techniques, it is easy to show that $\nabla u\in L^\infty_t W^{\beta,p_1}_x$ and one can then apply \cite[Lemma A.2 (ii)]{le2021taming} to obtain
\begin{align*}
\| (b^2-b^1)\cdot(I+\nabla u)\|_{L^{\tilde q_1}_t W^{-\beta,\tilde p_1}_x}
\lesssim \| b^1-b^2\|_{L^{\tilde q_1}_t W^{-\beta,\tilde p_1}_x} (1+ \| \nabla u\|_{L^\infty_t W^{\beta,p_1}_x})
\end{align*}
eventually leading to
\begin{equation}\label{eq:stability-variant}
\bigg\| \sup_{t\in[0,T]}|X_t^1-X_t^2|\bigg\|_{L^m_\omega}
    \lesssim \| X^1_0-X^2_0\|_{L^m_\omega} + \| b^1-b^2\|_{L^{\tilde q_1}_t W^{-\beta, \tilde p_1}_x}
    + \| \sigma^1-\sigma^2\Vert_{L^{\tilde q_2}_t L^{\tilde p_2}_x}
\end{equation}
for any choice of parameters $(\tilde p_1,\tilde q_1)\in\mathcal{J}_\beta$. This is in agreement with \eqref{eq:main-stab-1}, which is recovered for $\beta=0$.

\item $\beta> 1-2/q_1$. In this case we only know that $\nabla u\in L^{\bar q}_t W^{\beta, p_1}$ for $1/\bar{q}=1/q-(1-\beta)/2$; combined again with \cite[Lemma A.2 (ii)]{le2021taming}, this implies $(b^2-b^1)\cdot(I+\nabla u)\in L^{q'}_t W^{-\nu,\tilde p_1}_x$ for $1/q'=1/\bar{q}+1/\tilde q_1$ and the relevant condition becomes $(\tilde p_1,q')\in \mathcal{J}_\beta$. Overall we still obtain a stability estimate where $\| b^1-b^2\|_{L^{\tilde q_1}_t W^{-\beta,\tilde p_1}_x}$ appears, but now only under the more restrictive condition
\begin{equation}\label{eq:restrictive-parameter-stab}
\frac{2}{q_1}+\frac{2}{\tilde q_1}+\frac{d}{\tilde p_1}< 3-2\beta, \quad \frac{d}{p_1}<\beta.
\end{equation}
Condition \eqref{eq:restrictive-parameter-stab} agrees with $4/q_1+d/p_1<1$ needed to derive \eqref{eq:main-stab-2}, which corresponds to $\beta=1$, $\tilde p_1=p_1$ and $\tilde q_1= q_1$.
\end{enumerate}
\end{remark}

\section{Applications}\label{sec:applications}

\subsection{McKean--Vlasov equations}\label{subsec:MKV}

Throughout this section, for any $t\in [0,T]$ and any $X$, we will denote by $X^t_\cdot$ the stopped process $X^t_s:=X_{t\wedge s}$, which we ask the reader not to confuse with the pointwise evaluatation $X_t$.
We consider McKean--Vlasov (or distribution dependent) SDEs of the form
\begin{equation}\label{eq:MKV}
\dd X_t = B_t(X_t,\mu_t)\dd t + \Sigma_t(X_t,\mu_t)\dd W_t, \quad \mu_t=\cL(X^t_\cdot)
\end{equation}
for suitable coefficients $B,\,\Sigma$, see \cref{ass:MKV} below.

In the case of regular (e.g. Lipschitz) continuous functions $B,\,\Sigma$, equation \eqref{eq:MKV} can be solved classically; however in the presence of a singular drift $B$, several important results in the literature guaranteeing existence and uniqueness for \eqref{eq:MKV} require the diffusion term $\Sigma$  not depend on $\mu$, see e.g. \cite[Section 3.2]{mishura2020existence}, or anyway more restrictive assumption on $\Sigma$, cf. \cite[Section 4]{rockner2021well}; the reason lies in the use of transport-cost inequalities and Girsanov transform, see also \cite{lacker2018strong}. On the other hand there is another class of results, including \cite{huang2019distribution,huang2020mckean}, possibly involving the use of mixed Wasserstein-total variation distances, which naturally allows $\Sigma$ to depend on $\mu$; the shortcoming is a more restrictive assumption on $B$ and $\Sigma$, e.g. they must be Lipschitz in $\mu$ uniformly in $x$, which rules out basic cases of interest like $B(x,\mu)=b\ast \mu$ for unbounded $b$.
Let us stress that the literature on the topic is quite vast and there are many other works not properly fitting in this dichotomy, see for instance \cite{huang2021distribution,ren2022linearization, issoglio2021mckean};
finally, let us mention the remarkable work \cite{zhao2020distribution} for existence and uniqueness results, in the presence of $\mu$-dependent $\Sigma$, under yet another set of assumptions.

Here we do not follow any of the aforementioned approaches, instead we derive existence and uniqueness results for \eqref{eq:MKV} by readapting the strategy from \cite{galeati2022distribution}, based on the use of stability estimates in negative Sobolev spaces and Wasserstein distances.
Our assumptions cover the basic convolutional case $B(x,\mu)=b\ast \mu$ (cf. \cref{cor:MKV-convolutional}), but also allow for $\Sigma$ to depend on $\mu$, as well as a more general abstract class of coefficients $B$ and $\Sigma$ with nonlinear dependence on $\mu$ (cf. \cref{ass:MKV}).\\

In the following, given a separable Banach space $E$, we denote by $\cP(E)$ the set of probability measures on $E$ and by $\mathbb{W}_m$ the associated $m$-Wasserstein distance, namely
\[
\mathbb{W}_m(\mu,\nu)=\inf \bigg\{ \Big( \int \| x-y\|_E^m\, \pi({\rm d} x,{\rm d} y) \Big)^{1/m} : \pi\in \Pi(\mu,\nu)\bigg\}
\]
where $\Pi(\mu,\nu)\subset \cP(E\times E)$ denotes the set of couplings of $(\mu,\nu)$.
We refer to \cite{villani2009} for a complete survey on the properties of $\mathbb{W}_m$
; let us only recall that $\mathbb{W}_m(\mu,\nu)\leq \mathbb{W}_{\tilde m}(\mu,\nu)$ whenever $m\leq \tilde m$ and that for any $m\in [1,\infty)$, $\mu,\nu\in\cP(E)$ there exists an optimal coupling for $\mathbb{W}_m(\mu,\nu)$, which can be realised as a pair of $E$-valued random variables $(Y,Z)$ such that $\cL(Y)=\mu$, $\cL(Z)=\nu$ and $\mathbb{W}_m(\mu,\nu)=\mE[\| Y-Z\|_E^m]^{1/m}$.
We will mainly consider $E$ as the path space $C_T:=C([0,T];\R^d)$ or $E=\R^d$.

We are now ready to introduce our assumptions on the coefficients $B$, $\Sigma$.

\begin{assumption}\label{ass:MKV}
There exist $p_1,\,p_2 \in (d,\infty)$, $m\in [1,\infty)$ and $C>0$ such that the maps $B:[0,T]\times \R^d\times \cP(C_T)\to \R^d$, $\Sigma:[0,T]\times \R^d\times \cP(C_T)\to \R^{d\times d}$ satisfy the following:
\begin{itemize}
\item[i)] $\| B_t(\cdot,\mu)\|_{L^{p_1}_x} \leq C$, $\| B_t(\cdot,\mu)-B_t(\cdot,\nu)\|_{W^{-1,{p_1}}_x} \leq C\, \mathbb{W}_m(\mu,\nu)$ for all $t\in [0,T]$, $\mu,\nu\in \cP(C_T)$.
\item[ii)] $\| \nabla \Sigma_t(\cdot,\mu)\|_{L^{p_2}_x} \leq C$, $\| \Sigma_t(\cdot,\mu)-\Sigma_t(\cdot,\nu)\|_{L^{p_2}_x} \leq C\, \mathbb{W}_m(\mu,\nu)$ for all $t\in [0,T]$, $\mu,\nu\in \cP(C_T)$.
\item[iii)] $C^{-1} |\xi|^2 \leq | \Sigma(t,x,\mu)^\ast \xi|^2 \leq C |\xi|^2$ for all $\xi\in\R^d$, $t\in [0,T]$, $x\in \R^d$ and $\mu\in\cP(C_T)$.

\end{itemize}
\end{assumption}

Under \cref{ass:MKV}, one can give meaning to equation \eqref{eq:MKV} as follows: given a continuous adapted process $Y$, we can define $\mu_t=\cL(Y^t_\cdot)$ and then
\begin{align}
    \label{coe:b-sigma} b^Y_t(x):= B_t(x,\mu_t),\quad \sigma^Y_t(x):=\Sigma_t(x,\mu_t);
\end{align}
it can be readily checked that $(b^Y,\sigma^Y)$ satisfy \cref{main-ass}
\footnote{Since $\nabla \sigma^Y\in L^p_x$ for some $p>d$, Morrey's inequality implies that it is $\alpha$-H\"older in space for $\alpha=1-d/p$, uniformly in time; thus conditions concerning the modulus of continuity of $\sigma^X$ are automatically satisfied.}
, thus the associated SDE is wellposed.
We then define $X$ to be a solution to \eqref{eq:MKV} if and only if it is a solution to the SDE associated to $(b^X,\sigma^X)$; the concepts of strong existence, pathwise uniqueness and uniqueness in law for \eqref{eq:MKV} can then be readapted similarly.
Let us finally recall that, as usual for McKean--Vlasov SDEs, we will consider the initial datum $X_0$ to be random and independent of $W$; depending on how one wants to formulate it, the data of the problem can then be considered equivalently $(X_0,B,\Sigma)$ or $(\mu_0,B,\Sigma)$, where $\mu_0=\cL(X_0)$.

\begin{theorem}\label{thm:MKV}
Let $B$, $\Sigma$ satisfy \cref{ass:MKV} for some $m\in [1,\infty)$.
Then for any $\mu_0\in \cP_m(\R^d)$ strong existence, pathwise uniqueness and uniqueness in law holds for the SDE \eqref{eq:MKV}.
A similar statement holds for Stratonovich version of \eqref{eq:MKV}, i.e. with $\Sigma_t(X_t,\mu_t) \dd W_t$ replaced by $\Sigma_t(X_t,\mu_t) \circ \dd W_t$.
\end{theorem}

\begin{proof}
We give the proof in the It\^o case, the Stratonovich one being identical (up to applying Corollary \ref{cor:stratonovich} in place of Theorem \ref{thm:main-theorem}).
Standard arguments involving the structure of eq. \eqref{eq:MKV} (see for instance \cite[Propositions 3.5-3.7]{galeati2022continuous}) allow to show that any weak solution is a strong one and that pathwise uniqueness implies uniqueness in law; thus our task is reduced to show existence and uniqueness among strong solutions.
Consider the filtration $\mathcal{F}_t$ generated by $(X_0,W)$ and let $E$ be the space of continuous, $\mathcal{F}_t$-adapted paths $Y\in L^m_\omega C_T$, endowed with the metric
\[
d_E(Y,Z):=
\sup_{t\in [0,T]} e^{-\lambda t}\, \mE\Big[\sup_{s\leq t} |Y_s - Z_s|^m\Big]^{1/m}
= \sup_{t\in [0,T]} e^{-\lambda t}\, \| Y^t_\cdot - Z^t_\cdot\|_{L^m_\omega C_T}
\]
for a suitable $\lambda>0$ to be chosen later; define a map $\mathcal{I}$ from $E$ to itself, which to a given process $Y$ associates the unique strong solution $\mathcal{I}(Y)$ to the SDE associated to $(X_0,b^Y,\sigma^Y)$.
It is then clear that $X$ is a solution to \eqref{eq:MKV} if and only if $X=\mathcal{I}(X)$ and that our task is reduced to showing that $\mathcal{I}$ is a contraction on $(E,d)$, for suitable choice of $\lambda$.
To this end, let us fix a parameter $q\in (2,\infty)$ sufficiently large so that $4/q+d/p_1<1$, $2/q+d/p_2<1$; then by estimate \eqref{eq:main-stab-2} for the choice $(\tilde p_2,\tilde q_2)=(p_2,q)$ and \cref{ass:MKV}, for any $t\in [0,T]$ it holds
\begin{align*}
\| \mathcal{I}(Y)^t_\cdot -  \mathcal{I}(Z)^t_\cdot \|_{L^m_\omega C_T}^q
& \lesssim \| b^Y-b^Z\|_{L^q_t W^{-1,p_1}_x}^q + \| \sigma^Y-\sigma^Z\|_{L^q_t L^{p_2}_x}^q \\
& \lesssim \int_0^t \mathbb{W}_m(\mathcal{L}(Y^s_\cdot),\mathcal{L}(Z^s_\cdot))^q \dd s\\
& \lesssim \int_0^t e^{\lambda q s} d_E(Y,Z)^q \dd s \lesssim \lambda^{-1}\, e^{\lambda q t} d_E(Y,Z)^q;
\end{align*}
overall, this implies the existence of another constant $\tilde{C}>0$ such that
\[
d_E(\mathcal{I}(Y),\mathcal{I}(Z))\leq \tilde{C} \lambda^{-1/q} d_E(Y,Z)
\]
which implies contractivity of $\mathcal{I}$ once we choose $\lambda$ large enough.
\end{proof}

We now pass to provide sufficient conditions for \cref{ass:MKV} in a practical case of interest, given by \textit{convolutional drifts}; in particular here we take $\mu_t=\cL(X_t)\in \cP(\R^d)$ and consider maps $B,\Sigma$ with linear dependence on $\mu$ by means of a convolution.

\begin{lemma}\label{lem:MKV-1}
For any $p\in [1,\infty]$, $f\in L^p_x$ and $g$ such that $\nabla g\in L^p_x$, it holds
\begin{equation*}
\| f \ast (\mu-\nu)\|_{W^{-1,p}_x}\lesssim \| f\|_{L^p_x} \mathbb{W}_1 (\mu,\nu),
\quad \| g \ast (\mu-\nu)\|_{L^p_x}\lesssim \| \nabla g\|_{L^p_x} \mathbb{W}_1 (\mu,\nu)
\quad \forall\, \mu,\nu \in \mathcal{P}(\R^d).
\end{equation*}
\end{lemma}

\begin{proof}
Let us first prove the claim for $g$.
Recall the basic fact that, for any $h_i\in \R^d$, it holds
\[ \| g(\cdot + h_1)-g(\cdot+h_2)\|_{L^p_x} \leq \| \nabla g\|_{L^p_x} |h_1-h_2|.\]
Let $(Y,Z)$ be an optimal coupling for $(\mu,\nu)$, then by definition of convolution it holds
\begin{align*}
\| g \ast (\mu-\nu)\|_{L^p_x}
& = \big\| \mE[g(\cdot - Y)-g(\cdot -Z)] \big\|_{L^p_x}
\leq \mE\big[ \| g(\cdot - Y)-g(\cdot -Z)\|_{L^p_x} \big]\\
& \leq \|\nabla g\|_{L^p_x}\, \mE[ |Y-Z|] =  \|\nabla g\|_{L^p_x}\, \mathbb{W}_1(\mu,\nu) .
\end{align*}
The case of $f$ then follows from a duality argument.
Observe that $\langle f\ast (\mu-\nu), g\rangle=\langle f, g \ast (\tilde \mu-\tilde \nu)\rangle$, where $\tilde{\mu}(A)=\mu(-A)$ is the reflection of $\mu$; it is clear that $\mathbb{W}_1(\tilde\mu,\tilde{\nu})=\mathbb{W}_1(\mu,\nu)$, therefore
\begin{align*}
|\langle f\ast (\mu-\nu), g\rangle|
\leq \| f\|_{L^p_x} \| g\ast (\tilde{\mu}-\tilde{\nu})\|_{L^{p'}} 
\leq \| f\|_{L^p_x} \| g\|_{W^{1,p'}_x} \, \mathbb{W}_1(\mu,\nu)
\end{align*}
which yields the conclusion by taking the supremum over all $g\in W^{1,p'}_x$ with $\| g\|_{W^{1,p'}_x}=1$.
\end{proof}

\begin{lemma}\label{lem:MKV-2}
Let $\sigma$ satisfy conditions ($\mathcal{H}^\sigma_1$)-($\mathcal{H}^\sigma_3$) from \cref{main-ass} with $q_2=\infty$, $p_2>d$ and constant $K$ and let $\pi:[0,T]\times \R^d\to \R^{d\times d}$ satisfy
\[\nabla \pi\in L^\infty_t L^{p_2}_x,\quad \sup_{t,x} |\pi_t(x)| \leq (1-\delta) K^{-1/2}\]
for some $\delta\in(0,1)$. Then $\Sigma_t(x,\mu):= \sigma_t(x)- (\pi_t\ast \mu)(x)$ satisfies conditions {\rm ii)}-{\rm iii)} from \cref{ass:MKV}.
\end{lemma}

\begin{proof}
By standard properties of convolution, for any $\mu\in\cP(\R^d)$ it holds
\[ \| \nabla (\pi_t\ast \mu)\|_{L^{p_2}_x} \leq \| \nabla \pi_t\|_{L^{p_2}_x} \leq \| \nabla \pi\|_{L^\infty_t L^{p_2}_x} \]
%
%
which immediately implies a uniform bound on $\| \nabla \Sigma_t(\cdot,\mu)\|_{L^{p_2}_x}$; a similar argument shows uniform boundedness of $\Sigma_t(x,\mu)$.
By our assumptions and triangular inequality, it holds
\[
|\Sigma_t(x,\mu) \xi|
\geq |\sigma_t(x) \xi| - \| \pi\ast \mu\|_{L^\infty_{t,x}} |\xi|
\geq |\sigma_t(x) \xi| - \| \pi\|_{L^{\infty}_{t,x}}|\xi| \geq \delta K^{-1/2} |\xi|
\]
which shows uniform ellipticity.
\end{proof}

Combining \cref{lem:MKV-1} and \cref{lem:MKV-2} immediately implies the following

\begin{corollary}\label{cor:MKV-convolutional}
Let $b\in L^\infty_t L^{p_1}_x$ for some $p_1\in (d,\infty)$ and let $(\sigma,\pi)$ satisfy the hypothesis of \cref{lem:MKV-2}; consider
\[
B_t(X_t,\cL(X_t)):=(b_t\ast \cL(X_t))(X_t) \quad
\Sigma_t(X_t,\cL(X_t)):= \sigma_t(X_t)-(\pi_t\ast\cL(X_t))(X_t).
\]
Then strong wellposedness holds for the associated McKean--Vlasov SDE \eqref{eq:MKV}.
\end{corollary}

\begin{remark}\label{rem:MKV-1}
It is clear that other variants of \cref{thm:MKV} are allowed; for instance in \cref{ass:MKV}-{\rm i)} can be modified by requiring the existence of parameters $(q_1,p_1)$ satisfying $4/q_1+d/p_1<1$ and a function $\ell\in L^{q_1}_t$ such that
\[
\| B_t(\cdot,\mu)\|_{L^{p_1}_x} \ell_t \leq C,\quad \| B_t(\cdot,\mu)-B_t(\cdot,\nu)\|_{W^{-1,{p_1}}_x} \leq \ell_t \, \mathbb{W}_m(\mu,\nu).
\]
In the style of \cref{rem:interpolation-stability}, one could also formulate a condition on $B$ involving $W^{-\nu,p_1}_x$-norms for other values $\nu\in [0,1]$, although the choice $\nu=1$ is natural in view of \cref{lem:MKV-1}.
\end{remark}

\begin{remark}\label{rem:MKV-2}
Unfortunately, our result in general does not cover the choice $\Sigma_t(x,\mu):=(\sigma_t\ast\mu)(x)$; the problem is that, even if $\sigma$ is uniformly elliptic, the same doesn't need to hold for $\sigma\ast\mu$, as cancellations might happen inside the convolution.
This difficulty can be overcome under more information on $\sigma$; for instance if $\sigma_t(x)=\gamma_t(x) A$, where $\gamma:[0,T]\times \R^d\to \R$ and $A$ is a fixed, uniformly elliptic $\R^{d\times d}$ matrix, then the continuity and uniform ellipticity requirements on $\sigma$ imply that $\gamma_t(x)\geq \delta>0$ (or $\gamma_t(x)\leq -\delta<0$, which doesn't change the analysis).
Due to the positivity of $\mu$, we then find
\begin{align*}
|(\sigma_t\ast \mu)^\ast (x) \xi| = |\gamma_t\ast \mu(x)| |A^\ast \xi| \geq \delta |A^\ast \xi|
\end{align*}
showing again uniform ellipticity of $\Sigma_t(x,\mu)=(\sigma_t\ast \mu)(x)$.
\end{remark}

\begin{remark}\label{rem:MKV-3}
Other examples of $B$ and $\Sigma$, satisfying conditions similar to \cref{ass:MKV}, are discussed in \cite[Section 2.1]{galeati2022distribution}. Among them let us mention the case of dependence on statistics of $\mu$, e.g. $B_t(x,\mu)=b_t(x-\langle \phi,\mu\rangle)$ where $\langle \phi,\mu\rangle:=\int_{C_T} \phi(\omega) \mu(\d \omega)$ and $\phi:C_T\to \R^d$ is a uniformly Lipschitz map.
Indeed, by properties of translations in Sobolev spaces and Kantorovich-Rubinstein duality, it holds
\[
\| b_t(\cdot-\langle \phi,\mu\rangle)-b_t(\cdot-\langle \phi,\nu\rangle)\|_{W^{-1,p}_x}
\lesssim \| b_t\|_{L^p_x} |\langle \phi,\mu-\nu\rangle|
\lesssim \| b_t\|_{L^p_x} \| \phi\|_{Lip} \mathbb{W}_1(\mu,\nu).
\]
\end{remark}

\subsection{Strong compactness of solutions to singular SDEs}\label{subsec:compactness}

We show here that, under quite minimal assumptions, the solutions to SDEs of the form \eqref{eq:sde}, driven by the same reference Brownian motion $W$, form a compact set in the \textit{strong} topology of $L^m_\omega C_T$. This is much stronger than standard tightness results for $\mathcal{L}(X)$. 

Throughout this section we will assume for simplicity $X_0=x\in \R^d$ to be deterministic and fixed, so the dependence of the solution on it will not appear in the sequel. Given a pair $(b,\sigma)$ satisfying \cref{main-ass}, we denote by $X^{b,\sigma}$ the associated solution to the It\^o SDE \eqref{eq:sde} with initial condition $x$; similarly, $\tilde X^{b,\sigma}$ denotes the solution to the correposnding Stratonovich SDE.

We start by dealing with autonomous drift and diffusion; thus we fix some parameters $p_1,p_2\in (d,+\infty)$ corresponding to $b\in L^{p_1}_x$, $\nabla \sigma\in L^{p_2}_x$.
\footnote{As in \cref{subsec:MKV}, $\nabla \sigma\in L^{p_2}_x$ imples $\sigma$ being globally H\"older, so that we can drop assumption ($\mathcal{H}^\sigma_2$) and the dependence on a modulus of continuity $h$.} 

\begin{theorem}\label{thm:compactness}
Consider a family of coefficients $\mathcal{A}\subset\{(b,\sigma): b\in L^{p_1}_x,\, \sigma\in L^\infty_x,\, \nabla\sigma\in L^{p_2}_x\}$ satisfying \cref{main-ass} in a uniform way;
namely, there exist constants $M,\, K>0$ such that for all $(b,\sigma)\in\mathcal{A}$ it holds $\|b \|_{L^{p_1}_x} + \|\sigma\|_{L^\infty_x}+\| \nabla \sigma\|_{L^{p_2}_x}\leq M$ and $|\sigma^\ast \xi|\geq K^{-1} |\xi|^2$ for all $\xi\in\R^d$.
Then the families $\{X^{b,\sigma}:(b,\sigma)\in\mathcal{A}\}$ and $\{\tilde X^{b,\sigma}:(b,\sigma)\in\mathcal{A}\}$ are precompact compact in $L^m_\omega C_T$, for all $m\in [1,\infty)$.
\end{theorem}

\begin{proof}
We focus on the It\^o case, the Stratonovich one being almost identical up to applying \cref{cor:stratonovich} instead of \cref{thm:main-theorem}.
By well-known results concerning the SDE \eqref{eq:sde} and our assumptions, the solutions $X^{b,\sigma}$ satisfy uniform moment estimates, namely for all $m\in [1,\infty)$ it holds $\sup_{(b,\sigma)\in \mathcal{A}} \| X^{b,\sigma}\|_{L^m_\Omega C_T}<\infty$; thus it suffices to show precompactness w.r.t. convergence in probability.

Given any sequence $(b^n,\sigma^n)_n\subset\mathcal{A}$, by the assumptions and weak compactness we can extract a (not relabelled) subsequence such that $b^n$ converge weakly in $L^{p_1}_x$ to some $b$ and $\sigma^n$ converge weakly-$\ast$ in $L^\infty_x$ to some $\sigma$, with the additional property that $\nabla \sigma^n$ converge weakly to $\nabla \sigma$ in $L^{p_2}_x$. We claim that $X^{b^n,\sigma^n}$ converge in probability to $X^{b,\sigma}$, which yields the conclusion in the It\^o case.

From now on we will adopt the shorter notation $X^n=X^{b^n,\sigma^n}$, $X=X^{b,\sigma}$.
To show the claim, let us fix some $\delta>0$ and choose $R>0$ large enough so that $\mathbb{P}(\| X^n\|_{C_T} >R)<\delta$, $\mathbb{P}(\| X\|_{C_T} >R)<\delta$, which is possible by the uniform moment estimates;
next, let us choose $\psi\in C^\infty_c$ such that $\psi\equiv 1$ on $B_R$ and $\psi\equiv 0$ on $B_{R+1}$, where $B_R:=\{x\in \R^d: |x|\leq R\}$. Denote by $\bar{X}^n$ the solutions associated to $(\psi b^n, \psi \sigma^n)$, similarly $\bar{X}$.
Clearly $\psi b^n$ converge weakly to $\psi b$, which together with the compact embedding $L^{p_1}_x \hookrightarrow W^{-1,p_1}_{x,loc}$ and the presence of the cutoff $\psi$ implies $\| \psi b^n-\psi b\|_{W^{-1,p_1}_x} \to 0$ as $n\to\infty$; similarly $\| \psi \sigma^n -\psi \sigma\|_{L^{p_2}_x}\to 0$ and so by \cref{thm:main-theorem} we deduce that $\bar{X}^n\to \bar{X}$ in $L^m_\omega C_T$.
On the other hand, by construction $\bar{X}^n$ and $X^n$ coincide as long as $X^n$ do not get outside $B_R$, similarly for $X$, thus we obtain
\begin{align*}
\lim_{n\to\infty}\mathbb{P}(\| X^n-X\|>\eps)
& \leq \lim_{n\to\infty} \mathbb{P}(\| X^n\|>R)+ \mathbb{P}(\| X\|>R)+\mathbb{P}(\| \bar{X}^n-\bar{X}\|>\eps)\\
& \leq 2\delta + \lim_{n\to\infty} \mathbb{P}(\| \bar{X}^n-\bar{X}\|>\eps) = 2\delta.
\end{align*}
By the arbitrariness of $\delta>0$, the conclusion follows.
\end{proof}

In the case of time-dependent coefficients, we cannot go through the same proof, since local compactness in weak topologies is not obvious anymore; however, under mild additional time regularity assumptions, we can invoke the partial compactness results from \cite[Section 9]{simon1986compact}. In the next statement, we fix $(p_1,q_1)\in \mathcal{J}_1$ and consider $q_2=\infty$, $p_2\in (d,+\infty)$.

\begin{corollary}\label{cor:strong-compactness}
Consider a family $\mathcal{A}\subset\{(b,\sigma): b\in L^{q_1}_t L^{p_1}_x,\, \sigma\in L^\infty_{t,x},\, \nabla\sigma\in L^{\infty}_t L^{p_2}_x\}$ with the following property:
there exist some parameters $s>0$ and $N\in \mathbb{N}$, as well as constants $M,\, K>0$ such that for all $(b,\sigma)\in\mathcal{A}$ it holds
\footnote{Here we follow the notation from \cite{simon1986compact}; for a Banach space $E$, the space $W^{s,1}_t E= W^{s,1}(0,T;E)$ is defined as the set of integrable functions $f:[0,T]\to E$ such that
\[ \| f\|_{W^{s,1}_t E} := \| f\|_{L^1_t E} + \llbracket f\rrbracket_{W^{s,1}_t E}<\infty, \quad \text{where} \quad  \llbracket f\rrbracket_{W^{s,1}_t E} := \int_{[0,T]^2} \frac{\| f_r-f_u\|_E}{|r-u|^{1+s}} \dd r\dd u.\]}
\begin{equation*}
\| b\|_{L^{q_1}_t L^{p_1}_x} + \| b\|_{W^{s,1}_t W^{-N,p_1}_x} + \|\sigma\|_{L^\infty_{t,x}}
+ \| \nabla \sigma\|_{L^\infty_t L^{p_2}_x} + \| b\|_{W^{s,1}_t W^{-N,p_2}_x} \leq M
\end{equation*}
and $|\sigma^\ast \xi|\geq K^{-1} |\xi|^2$ for all $\xi\in\R^d$.
Then the families $\{X^{b,\sigma}:(b,\sigma)\in\mathcal{A}\}$ and $\{\tilde X^{b,\sigma}:(b,\sigma)\in\mathcal{A}\}$ are precompact compact in $L^m_\omega C_T$, for all $m\in [1,\infty)$.
\end{corollary}

\begin{proof}
Introducing the cutoff $\psi$ as in the proof of \cref{thm:compactness}, we can always assume $(b,\sigma)$ to be compactly supported on $B_R$; thus the only relevant issue is to show the existence of a subsequence $(b^n,\sigma^n)$ converging to $(b,\sigma)$ in a suitable weak topology, which still has to be strong enough in order for our stability estimates to apply.
For any $\beta>0$, we can apply \cite[Corollary 7]{simon1986compact} for the choice $X=L^{p_1}_x$, $B=W_x^{-\beta,p_2}$ and $Y=W_x^{-\beta-N,p_2}$ to deduce the existence of a subsequence such that $b^n\to b$ in $L^{\tilde q_1}_t W^{-\beta,p_1}_x$ for any $\tilde{q}_1<q_1$, where $b\in L^{q_1}_t L^{p_1}_x$; similarly $\sigma^n\to \sigma$ in $L^{\tilde q_2}_t L^{p_2}_x$ for any $\tilde{q}_2<+\infty$, where $\sigma$ still satisfies the same bounds as the elements from $\mathcal{A}$. In particular, the SDE associated to $(b,\sigma)$ is well defined, with solution $X^{b,\sigma}$.
Since $(p_1,q_1)\in\mathcal{J}_1$ and $p_2>d$, we can always find $\tilde q_1,\, \tilde q_2$ and $\beta$ satisfying the above and such that additionally $\beta<1-2/q_1$, $(\tilde q_1,p_1)\in\mathcal{J}_\beta$ and $(\tilde q_2,p_2)\in\mathcal{J}_1$, which by the stability estimate \eqref{eq:stability-variant} implies $X^{b_n,\sigma_n}\to X^{b,\sigma}$ and thus the conclusion in the It\^o case.

The reasoning for Stratonovich SDEs is almost identical, but we give some additional details since we haven't precisely stated a version of \eqref{eq:stability-variant} in this case. By rewriting the SDE in the corresponding It\^o form, it suffices to verify that $\sigma^n\cdot\nabla \sigma^n$ converge strongly to $\sigma\cdot \nabla \sigma$ in suitable topologies; by \cite[Corollary 7]{simon1986compact}, we can show $\sigma^n\to \sigma$ in $L^{\tilde q_2}_t W^{\beta, p_2}_x$ for any $\beta<1$, thus $\sigma^n\to \sigma$ in $L^{\tilde q_2}_t W^{\beta-1, p_2}_x$ and by \eqref{eq:product-distributions} we can deduce that $\sigma^n\cdot\nabla \sigma^n \to \sigma\cdot\nabla \sigma$ in $L^{\tilde q_2/2}_t W^{\beta-1, p_2}_x$. As the argument holds for any $\tilde q_2<\infty$ and any $\beta<1$ and by assumption $p_2>d$, we can always choose such parameters in a way that $(\tilde q_2/2, p_2)\in\mathcal{J}_{1-\beta}$, yielding the conclusion again by an application of \eqref{eq:stability-variant}.
\end{proof}

\begin{remark}\label{rem:compactness-conjecture}
We conjecture that the time regularity condition, in the form of a uniform bound in $W^{s,1}_t W^{-N,p}_x$, is not needed in order to derive the strong compactness result of \cref{cor:strong-compactness}.
In order to avoid it, one would need a stability estimate where the difference of differnt coefficients is measured in a negative norm of both time and space, e.g. $\| b^1-b^2\|_{W^{-s,q_1}_t W^{-\beta,p_1}}$.
In this direction, let us mention that it is possible to obtain estimates for integrals of the form $\int_t f_r(X_r) \dd r$, by means of \textit{nonlinear Young integrals} (cf. \cite{catellier2016averaging, hu2017nonlinear}), for instance when $f\in C^{-\gamma}_t C^1_x$ for $\gamma<1/2$;
combining this result with \cref{lemm-no-drift}, giving instead estimates when $f$ has positive regulartiy in time but negative in space, and interpolation techniques, there is some hope to achieve estimates for $f$ enjoying (possibly very small) negative regularity in both variables.
We leave this problem for future investigations.
\end{remark}

\subsection{Improved Wong-Zakai Theorem}\label{subsec:wong-zakai}
In this section we apply our stability results in the context of Wong--Zakai approximation of singular SDEs.
The main difficulty lies in the fact that, due to the singularity of the drift $b$, the SDE associated to a smooth approximation $W^n$ in place of $W$ does not need to be wellposed, making even the correct formulation of the result unclear.
To overcome this difficulty, we follow the strategy introduced in  \cite{ling2021wong}, based on first replacing the coefficients $(b,\sigma)$ by some smooth approximations of them and then applying the Wong--Zakai result therein; one thus needs to control the overall error committed in this two-step approximation.
Compared to \cite{ling2021wong}, we improve the result from by not only allowing $b$ to be singular, but also considering $\sigma$ only Sobolev differentiable and not $C^2_x$.


We start by introduction some basic notations and conventions for this section.\\
For simplicity, here we will only consider autonomous coefficients $(b,\sigma)$ satisfying \cref{main-ass}, as well as deterministic initial data $x_0\in\R^d$.
We always work on the Wiener space $(\Omega,\mP,\mathcal{F},(\mathcal{F}_t)_{t\geq0})$, that is to say, $\Omega:=\{\omega\in\mathcal{C}([0,\infty),\mathbb{R}^d):\omega(0)=0\}$, $\mP$ is the Wiener measure defined on the Borel $\sigma$-algebra $\mathcal{F}$ of $\Omega$, $W_t(\omega):=\omega(t),\,t \geq 0$ is a Wiener process and $(\mathcal{F}_t)_{t\geq0}$ is the filtration generated by  $(W_t)_{t \geq 0}$.

We shall consider the following class of approximations:

\begin{definition}\cite[VI. Definition 7.1]{IkedaWatanabe}\label{apw}
By an approximation of the Wiener process $W_t$, we mean a family $\{W^n\}_{n\geq0}$ of $d$-dimensional continuous processes defined on the Wiener space $(\Omega,\mP,\mathcal{F},(\mathcal{F}_t)_{t\geq0})$ such that
\begin{itemize}
    \item[(1)] for every $\omega\in\Omega, t\mapsto W^n_t(\omega)$ is continuous and piecewise continuously differentiable;
    \item[(2)] $W^n_0(\omega)$ is $\mathcal{F}_{1/n}$-measurable and $\mE W^{n}_0=0$;
    \item[(3)] $W^n_{t+k/n}(\omega)=W^n_{t}(\theta_{k/n}\omega)+\omega(k/n)$, for every $k\in\mathbb{N}$, $t\geq 0$ and $\omega\in\Omega$, where $\theta_t$ is the shift operator defined by $(\theta_t\omega)(s)=\omega(t+s)-\omega(t)$;
    \item[(4)] there exists a positive constant $C$ such that
    $$\mE\big[|W^n_{0}|^6\big]\leq C n^{-3},\quad \mE\bigg( \int_0^{1/n}|\dot W^n_s\footnote{$\dot W^n_s:=\frac{\dd W^n_s}{\dd s}$}|\dd s\bigg)^6\leq C n^{-3}.$$
\end{itemize}
\end{definition}

\noindent From \cite[VI. Section 7]{IkedaWatanabe} we know that for such a approximation sequence $W^n$ we have 
$$\lim_{n\rightarrow\infty}\mE\bigg[ \sup_{0\leq t\leq T}|W_t-W_t^n|^2 \bigg]=0$$
for every $T>0$.
We further introduce the following notations for $t>0$ and $i,j\in \{1,\ldots, d\}$:
\begin{align}\label{sijn}
s_{ij}^n(t):=s_{ij}(t,n)
:= \frac{1}{2t} \,\mE\Big[ \int_0^t \big( W^{n}_{s,i} \dot W^{n}_{s,j}-W^{n}_{s,i} \dot W^{n}_{s,i} \big)\dd s \Big]
\end{align}
and
\begin{align}\label{cijn}
c_{ij}^n(t):=c_{ij}(t,n):=\frac{1}{t}\, \mE\Big[\int_0^t \dot W^{n}_{s,i} (W^{n}_{t,j}-W^{n}_{s,j}) \dd s\Big]
.
\end{align}
%
Notice that $s(t,n)$ is a skew-symmetric $d\times d$-matrix for each $t$ and $n$, i.e. $s_{ij}(t,n)=-s_{ji}(t,n)$. In order to control the convergence rate of the solutions $X^n$ driven by $W^n$ to $X$ driven by $W$, we will impose the following:

\begin{assumption}\label{asssc}
There exists a skew-symmetric $d\times d$-matrix $s$ and a rate function $f_n$ such that
\begin{align*}
\big|s_{ij}^n( n^{-1})-s_{ij}\big|\leq f_n,\quad \lim_{n\rightarrow\infty}f_n=0,\quad i,j=1,\ldots,d.
\end{align*}
\end{assumption}

\noindent Under \cref{asssc}, we define
\begin{align}\label{cij}
c_{ij}:=s_{ij}+\frac{1}{2}\delta_{ij},\quad i,j=1,\cdots,d.
\end{align}
Correspondingly, for any pair of functions $\varphi,\psi:\R^d\to \R^{d\times d}$ sufficiently regular, we define a vector field $c:\varphi\cdot\nabla \psi:\R^d\to \R^d$ componentwise by
\[
\big(c:\varphi \cdot\nabla \psi\big)_k(x) := \sum_{i,j,l=1}^d c_{ij}\, \varphi_{il}(x)\partial_l \psi_{jk} (x)\quad \forall\, k=1,\ldots,d.
\]
%
%

The next result gives us the dependence between the convergence rate of $X^n$ to $X$ and the coefficients $b$, $\sigma$; it is just a useful rewriting of \cite[IV. Theorem 7.2]{IkedaWatanabe}, making explicitly the dependence on the coefficients $(b,\sigma)$ and the rate of convergence. For more details on its derivation, we refer the interested reader to \cref{sec:Wong-Zakai}.

\begin{theorem}[Wong--Zakai]\label{smoothdrift}
Let $b\in C_x^1$,  $\sigma\in C_x^2$ and $(W^n_t)_{n\geq 1}$ be an approximation of $W_t$, in the sense of \cref{apw}, satisfying \cref{asssc}. Consider the SDEs
\begin{align*}
X_t=x_0+\int_0^tb(X_s) \dd s + \int_0^t\sigma(X_s) \dd W_s +  \int_0^t (c:\sigma\cdot\nabla \sigma)(X_s) \dd s,  
\end{align*}
and
\begin{align*}
X_t^{n} = x_0 + \int_0^tb(X_s^{n}) \dd s + \int_0^t\sigma(X_s^{n}) \dd W_{s}^n.
\end{align*}
Then for each $T>0$ there exists a constant $C$ such that
\begin{align}\label{errorest}
\mE\Big[ \sup_{t\in[0,T]}|X_t-X_t^{n}|^2\Big]\leq \exp\Big( C \big(1 + \Vert b\Vert_{C_x^1}^2 + \| \sigma\|_{C^1_x}^4 + \| \sigma\|_{C^0_x}^2\, \Vert\sigma\Vert_{C_x^2}^2 \big) \Big) (f^2_n+ n^{-\frac{1}{5}}),
\end{align}
where $f_n$ is from \cref{asssc} and $C$ depends on $T$ and $\max_{i,j} |c_{ij}|$, but not on $b$, $\sigma$ and $n$.
\footnote{It is claimed in \cite[Theorem 3.9]{ling2021wong} that the resulting error in \eqref{errorest} is $f_n^2+ n^{-1+\eps}$, where $\eps$ can be chosen arbitrarily small; unfortunately this is not the case, as there are other error terms of the form $n^{-\eps}$ which become too slow if $\eps$ is chosen too small. For more details, see \cref{sec:Wong-Zakai}.}
\end{theorem}

In order to control the error we commit by applying \cref{smoothdrift} at the level of smooth approximations of our coefficients, we introduce the following classes for $b$ and $\sigma$. 

\begin{definition}\label{bph}
Let $b\in L^{p_1}_x$ for some $p_1\in (d,+\infty)$, $p_1\geq 2$, and $r\in[0,\infty)\rightarrow[0,\infty)$ be an increasing function.
For a sequence $(b^n)_{n \geq 1}$ of continuous functions $b^n \colon \mathbb{R}^d \to \mathbb{R}^d$,
we write $(b^n)_{n\geq 1}\in \mathcal{A}(b,p_1,r)$ if:
\begin{itemize}
\item[(i)] $b^n\in C_x^1$, 
\item[(ii)]  $\Vert b-b^n\Vert_{L_x^{p_1}}\rightarrow0$ as $n\rightarrow\infty$,
\item[(iii)] $\Vert b^n\Vert_{C_x^1}\leq r(n)\Vert b\Vert_{L_x^{p_1}}$.
\end{itemize}
\end{definition}

\begin{definition}\label{sigma}
Let $p_2\in (d,+\infty)$, $p_2\geq 2$, and let $\sigma$ satisfy condition $(\mathcal{H}^\sigma_1)$ from \cref{main-ass} with constant $K$, $\nabla \sigma\in L^{p_2}_x$; let $r\in[0,\infty)\rightarrow[0,\infty)$ be an increasing function.
For a sequence $(\sigma^n)_{n \geq 1}$  of regular functions $\sigma^n \colon \mathbb{R}^{d\times d} \to \mathbb{R}^d$, we write $(\sigma^n)_{n\geq 1}\in \mathcal{A}(\sigma,p_2,r)$ if:
\begin{itemize}
\item[(i)] $\sigma^n$ satisfy condition $(\mathcal{H}^\sigma_1)$ with the same constant $K$, uniformly in $n$,
\item[(ii)] $\sigma^n\in C_x^2$,
\item[(iii)]  $\Vert \sigma-\sigma^n\Vert_{L_x^{p_2}}\rightarrow0$ as $n\rightarrow\infty$,
\item[(iv)] $\sup_{n\in\mathbb{N}} \| \nabla \sigma^n\|_{L^{p_2}_x}<\infty$ and $\| \sigma^n\|_{C^1_x}^2 + K \Vert \sigma^n\Vert_{C_x^2}\leq r(n)\Vert \sigma\Vert_{L_x^{p_2}}$.
\end{itemize}
\end{definition}

\noindent We are now ready to state our result.

\begin{theorem}\label{improved-wong-zakai}
Let $(p_1,p_2)\in (d,+\infty)^2$, $p_i\geq 2$, and let
$b \in L_x^{p_1}$, $\sigma$ satisfy condition $(\mathcal{H}^\sigma_1)$ in \cref{main-ass} and $\nabla \sigma\in L^{p_2}_x$.
Assume further that $(W^n)_{n \geq 1}$ is an approximation of $W$, in the sense of  \cref{apw}, satisfying \cref{asssc}. 
Let $X$ be the solution to
\begin{align}\label{wzksds}
X_t=x_0+\int_0^tb(X_s) \dd s+\int_0^t\sigma(X_s) \dd W_s + \int_0^t (c:\sigma\cdot\nabla \sigma)(X_s) \dd s
\end{align}
for $c$ as defined in \eqref{cij}.
Let $(b^n)_{n\geq 1}\in \mathcal{A}(b,p_1,r)$, $(\sigma^n)_{n\geq 1}\in \mathcal{A}(\sigma,p_2,r)$ be approximation sequences for $b$ and $\sigma$ such that
\begin{align}\label{hfn}
 \lim_{n\rightarrow\infty} \exp\Big(C r(n)^2 (\Vert b\Vert_{L_x^{p}}^2+\Vert \sigma\Vert_{L_x^{p}}^2)\Big) (f^2_n + n^{-1/5})=0
\end{align}
where $C$ is the constant appearing in \eqref{errorest} and $f_n$ is given in \cref{asssc}. Then we have
\begin{align}\label{error}
\lim_{n\rightarrow\infty} \mE\Big[\sup_{0\leq t\leq T}|X_t-X_t^{n}|^2\Big]=0,
\end{align}
where $(X_t^{n})_{t\geq0}$ is the solution to the equation
\begin{align*}
X_t^{n}=x_0+\int_0^t b^n(X_s^{n}) \dd s + \int_0^t\sigma^n(X_s^{n}) \dot W_{s}^n \dd s.
\end{align*}
\end{theorem}

\begin{proof}
Observe that, under our assumptions,  $c:\sigma\cdot\nabla\sigma\in L_x^{p_2}$, therefore the SDE \eqref{wzksds} is well-posed (either by \cite{XXZZ2020} or the facts recalled in \cref{subsec:preliminaries}).
Let $X_{m,t}$ be the solution to
\begin{align*}
X_{m,t}=x_0+\int_0^t b^m(X_{m,s})\dd s +\int_0^t \sigma^m(X_{m,s})\dd W_s
+ \int_0^t (c:\sigma^m\cdot\nabla\sigma^m)(X_{m,s})\dd s.
\end{align*}
By \cref{smoothdrift}, we know that for any $m\geq1$ it holds
\begin{align*}
 \mE\Big[\sup_{0\leq t\leq T}|X_{m,t}-X_{m,t}^{n}|^2\Big]
 & \lesssim \exp\Big( C (\Vert b^m\Vert_{C^1_x}^2 + \| \sigma\|_{C^1_x}^4 + \| \sigma\|_{C^0_x}^2 \Vert \sigma^m\Vert_{C^2_x}^2)\Big) (f^2_n + n^{-1/5})\\
 & \lesssim \exp\Big( C\, r(m)^2(\Vert b\Vert_{L_x^{p}}^2+\Vert \sigma\Vert_{L_x^{p}}^2) \Big) (f^2_n + n^{-1/5}),
\end{align*}
where $C$ is independent of $m$ and $n$ and $X_{m,t}^{n}$ is the solution to
\begin{align*}
X_{m,t}^{n}=x_0+\int_0^t b^m(X_{m,s}^{n})\dd s+\int_0^t\sigma^m(X_{m,s}^{n})\dot W_{s}^n \dd s.
\end{align*}
By choosing $m=n$  and using the condition \eqref{hfn}, we obtain
\begin{align*}
\lim_{n\rightarrow\infty} \mE\Big[\sup_{0\leq t\leq T}|X_{n,t}-X_{n,t}^{n}|^2\Big]=0.
\end{align*}
On the other hand, arguing similarly to the proof of \cref{cor:stratonovich}, using the estimates
\begin{align*}
    & \| c:(\sigma^m-\sigma)\cdot\nabla \sigma^m\|_{L^{p_2/2}_x}
    \lesssim \| \sigma^m-\sigma\|_{L^{p_2}_x} \| \sigma^m\|_{L^{p_2}_x}
    \lesssim \| \sigma^m-\sigma\|_{L^{p_2}_x},\\
    & \| c: \sigma\cdot\nabla (\sigma^m-\sigma)\|_{W^{-1,p_2}_x}
    \lesssim \| \nabla (\sigma^m-\sigma)\|_{W^{-1,p_2}_x} ( \| \sigma\|_{L^\infty_x} + \| \nabla \sigma\|_{L^{p_2}_x})
    \lesssim \| \sigma^m-\sigma\|_{L^{p_2}_x},
\end{align*}
it is not difficult to show that
\begin{align*}
    \mE\Big[\sup_{0\leq t\leq T}|X_t-X_{n,t}|^2\Big]
    \lesssim \Vert b-b^n\Vert_{L_x^{p_1}}^2  + \Vert \sigma-\sigma^n\Vert_{L_x^{p_2}}^2.
\end{align*}
Overall we obtain
\begin{align*}
\mE\Big[\sup_{0\leq t\leq T}|X_t-X_t^{n}|^2\Big]
&\lesssim \Vert b-b^n\Vert_{L_x^{p_1}}^2+\Vert \sigma-\sigma^n\Vert_{L_x^{p_2}}^2 +  \mE\Big[\sup_{0\leq t\leq T}|X_{n,t}-X_{n,t}^{n}|^2\Big]
\end{align*}
which implies \eqref{error} by letting $n\rightarrow\infty$. 
\end{proof}

\subsection{Other applications of stability estimates in the literature}\label{subsec:other}

We briefly discuss here two further examples of implicit applications of the stability estimates from \cref{stab-general} in the literature. In this case we do not provide any new results, yet we believe that highlighting this step makes the arguments from the papers \cite{le2021taming} and \cite{hao2022strong} clearer.

\subsubsection{Euler-Maruyama schemes for singular SDEs}\label{subsec:euler-maruyama}

Similarly to \cref{subsec:wong-zakai}, when dealing with singular drifts one cannot try to simulate the solution to the SDE directly, as the unboundedness of $b$ can be a major source of errors.
Again, one can circumvent this problem by introducing a two step approximation, based on first mollifying $b$, and then applying an Euler--Maruyama scheme to the SDE with approximate drift $b^n$. 
This idea has seen precursors in \cite{kohatsu2017weak}, in the study of weak error rates for integrable drifts, and in \cite{de2019numerical}, in the one-dimensional setting with $b$ only enjoying some negative H\"older regularity.
Here we follow the recent work \cite{le2021taming}, which studied the strong error in the general multidimensional setting, with $(b,\sigma)$ satisfying \cref{main-ass}; therein this two-step approximation was directly unified in a unique way and thus stability estimates were only used in a very implicit way.

Let $(b,\sigma)$ satisfy \cref{main-ass} and $X$ be the solution to \eqref{eq:sde}; there are two different (yet similar) ways one can proceed.

On one hand, given a sequence of smooth bounded approximations $(b^n)_n$ of $b$, one can introduce the solution $\tilde{X}^n$ associated to $(b^n,\sigma)$ and then apply the Euler--Maruyama scheme of step $n$ to $(b^n,\sigma)$, yielding a numerical approximation $\tilde{X}^{n,n}$;
the overall error splits into $e=e^1+e^2$, associated respectively to $X^n-\tilde X^n$ and $\tilde X^n-\tilde{X}^{n,n}$.
The first error $e^1$ can be estimated by means of either \cref{thm:main-theorem} or the more refined estimates from \cref{subsec:main-proof} and yields a quantity in the style of $\varpi_n$ as defined in \cite{le2021taming}.
The error $e_2$ instead is the purely numerical one, which has to be treated by different techniques and will result in a polynomial decay rate $n^{-\alpha}$.
Alternatively, one can proceed more directly as in \cite{le2021taming} and, given the approximations $(b^n)_n$, compare $X$ to the solutions $X^n$ to
%
\begin{align}\label{eq:Euler-Maruyama}
    \dd X_t^n=b^n_t(X_{k_n(t)}^n)\dd t+\sigma_t(X_{k_n(t)}^n)\dd W_t,\quad X^n|_{t=0}=X_0^n,
\end{align}
where $	k_n(t)= \lfloor tn \rfloor/n$, $\lfloor \cdot \rfloor$ denoting the integer part.
For fixed $n\in\mathbb{N}$, we are in the setting of \cref{stab-general} by considering $Y_t:=X_t^n$, corresponding to \eqref{eq:perturbed-sde} for the choice
$$R_t^1:=b^n_t(X^n_{k_n(t)})-b_t(X^n_t),\quad \quad R_t^2:=\sigma_t(X^n_{k_n(t)})-\sigma_t(X^n_t).$$
Assuming for the moment that $(Y,R^1,R^2)$ satisfy \cref{ass:perturbation}, an application of \cref{stab-general} then yields
\begin{align*}
  \Big\| \sup_{t\in [0,T]} |X_t-X^n_t| \Big\|_{L_\omega^m} \lesssim 
  \| X_0-X_0^n\|_{L_\omega^m} + \Big\| \sup_{t\in [0,T]} | V_t| \Big\|_{L_\omega^{m\gamma}}  
\end{align*}
for $V$ as defined in \eqref{eq:process-V}. By our assumptions, the properties of $u$ and martingale inequalities, it is easy to deduce that
\begin{equation}\label{eq:numerical-error}\begin{split}
    \Big\| \sup_{t\in [0,T]} |V_t|\Big\|_{L^m_\omega} 
    & \lesssim \Big\| \int_0^T |(b^n-b)_t(X^n_t)| \dd t \Big\|_{L^{m\gamma}_\omega} + \Big\| \int_0^T |(b^n_t(X^n_t)-b^n_t(X^n_{k_n(t)})| \dd t \Big\|_{L^{m\gamma}_\omega}\\
    & \quad + \Big\| \int_0^T |(\sigma_t(X^n_t)-\sigma_t(X^n_{k_n(t)})| |D^2 u_t(X^n_t)|\dd t \Big\|_{L^{m\gamma}_\omega}\\
    & \quad + \Big\| \Big( \int_0^T |(\sigma_t(X^n_t)-\sigma_t(X^n_{k_n(t)})|^2 \dd t\Big)^{1/2} \Big\|_{L^{m\gamma}_\omega}.
\end{split}\end{equation}
Again, the first term gives rise to the approximation error $e^1$ related to $\varpi_n$, while all the others are numerical errors producing a decay $n^\alpha$.
This is the point where the fine analysis on the properties of the process $X^n$ conducted in \cite[Section 5]{le2021taming}, also based on novel tools like stochastic sewing techniques, must kick in.
The same kind of results in fact verify that $(Y,R^1,R^2)$ verify \cref{ass:perturbation}, see for instance Theorem 5.1 and Lemma 5.14 from \cite{le2021taming}.

Nonetheless, we find it important to stress how reaching an estimate of the form \eqref{eq:numerical-error} can be performed in the framework of well established SDE tools and do not need any substantially new techniques.

\subsubsection{Interacting particle systems}
We now set ourselves in the setting of \cite{hao2022strong}, which studied the propagation of chaos for interacting particle systems with singular kernel of Krylov-R\"ockner type. Therein again \cref{stab-general} was applied, although implicitly; we make this point clear as follows.

For any $N\in\mathbb{N}$, consider the interacting particle system
\begin{align}\label{eq:particle-equ}
    \dd X_t^{N,i}=B_t(X_t^{N,i},\mu_t^N)\dd t+ \sigma_t(X_t^{N,i})\dd W_t^i,,\quad \mu_t^N:=\frac{1}{N}\sum_{i=1}^N\delta_{X_t^{N,i}}\quad \forall\,i=1,\cdots,N,
\end{align}  
where $X^{N,i}_0=\xi^i$ for some i.i.d. $\{\xi^i\}_{i\in\mathbb{N}}$, $\{W^i\}_{i\in\mathbb{N}}$ are independent Brownian motions, $\sigma$ satisfies properties $(\mathcal{H}^\sigma_1)-(\mathcal{H}^\sigma_3)$ from \cref{main-ass} and $B$ is of the form considered in \cite{hao2022strong}.
As $N\to\infty$, each particle $X^{i,N}$ of system \eqref{eq:particle-equ} is expected to converge to an independent copy of the associated McKean--Vlasov SDE
\begin{equation}\label{eq:MKV-interacting}
\dd X_t^i = B_t(X^i,\mu_t)\dd t + \sigma_t(X^i_t)\dd W^i_t,\quad \mu_t=\mathcal{L}(X^i_t),
\end{equation}
a property which is usually referred to as propagation of chaos; in \cite{hao2022strong} the authors show that such convergence holds in a strong sense, for suitable singular $B$, see Assumption \textbf{$({\rm H})^b$} therein.

To explain how their techniques are related to ours, let us first recall some basic facts.
Under \textbf{$({\rm H})^b$}, equation  \eqref{eq:MKV-interacting} is wellposed and the associated drift $b^\mu_t(x):=B_t(x,\mu_t)$ satisfies an $L^{q_1}_t L^{p_1}_x$ integrability condition
\footnote{Actually, the authors in \cite{hao2022strong} use the local spaces $L^{q_1}_t \tilde{L}^{p_1}_x$, but at the level of exposition this doesn't change much.}.
Moreover by exchangeability of \eqref{eq:particle-equ}, it suffices to show convergence for $i=1$.
Finally, by setting $X^{N,1}=Y$ and
$$R_t^1:=B_t(X^{N,1}_t,\mu^N_t)-b^\mu_t(X^{N,1}_t), \quad R_t^2:=0,$$
we are in the setting of \cref{stab-general}.
The triple $(Y,R^1,R^2)$ satisfies \ref{ass:perturbation} thanks to an application of the \textit{partial Girsanov transform}, cf. \cite[Lemma 5.3 (ii)]{hao2022strong}; observe that $X^{1,N}$ is measurable w.r.t. to the filtration generated by $\{\xi^i, W^i\}_{i=1}^N$, which is covered by \cref{ass:perturbation}-iii) by taking $Z=(\xi^i,W^i)_{i=2}^N$.
Therefore by \cref{stab-general} and the usual properties of $u$ we deduce that
\begin{align*}
\Big\| \sup_{t\in [0,T]} |X^{N,1}_t-X^1_t| \Big\|_{L^m_\omega} \lesssim \Big\| \int_0^T |B_t(X^{N,1}_t,\mu^N_t)-b^\mu_t (X^{N,1}_t)| \dd t\Big\|_{L^{m\gamma}_\omega}
\end{align*}
The convergence of the quantity on the r.h.s. to $0$ as $N\to \infty$ is established in \cite[Lemma 6.1]{hao2022strong} for $m\gamma=2$ (but the other cases follow similarly), which yields the desired propagation of chaos property.

Let us point out two interesting facts coming from our analysis:
\begin{itemize}
\item[i)] In order to apply our stability results, although some key information on the processes $X^{N,i}$ is needed, the main requirement is the regularity of $b^\mu$. This can be better than the one of the original $B$, especially in the case of convolutional drifts $B_t(x\mu)=(\phi_t\ast\mu)(x)$, thus there is hope to establish such results also in cases where $\phi$ does not satisfy the Krylov--R\"ockner condition, but $b^\mu$ does.
\item[ii)] We took $\sigma$ measure independent as in \cite{hao2022strong}, but this doesn't really play a role in our analysis; thus it might be possible to extend several result to the $\mu$-dependent diffusion as well. The key problem would then be to establish some analogue of \cite[Lemma 6.1]{hao2022strong} in this setting.
\end{itemize}


\appendix
\section{More details on the proof of \cref{smoothdrift}}\label{sec:Wong-Zakai}

%
In this appendix, we follow carefully the proof of \cite[IV. Theorem 7.2]{IkedaWatanabe} in order to indicate the precise dependence on the regularity of the coefficients $(b,\sigma)$ and the resulting convergence as stated it \cref{smoothdrift}.

It will be convenient in several passages to keep the same notations as in \cite{IkedaWatanabe}, based on the use of a parameter $\delta\ll 1$ and a function $N(\delta): (0,1]\to \mathbb{N}$ satisfying the hypothesis of \cite[IV. Lemma 7.1]{IkedaWatanabe}, so let us explain shortly how it translates into our framework.
The parameter $\delta$ corresponds to $n^{-1}$, while the function $N(\delta)$ is given by $N(\delta):=\lfloor \delta^{-\eps/4} \rfloor$ for a suitable parameter $\eps\in (0,1)$ to be chosen later on.
Observe that by construction, for any $\eps\in (0,1)$ it holds
\begin{align}\label{ndelta}
     \lim_{\delta \to 0^+} N(\delta)^4\delta =0, \quad \lim_{\delta\to 0} N(\delta)=+\infty.
\end{align}
For such choice of $N(\delta)$, we set $\tilde \delta:=N(\delta)\delta$ and
$$\quad [s]^+(\tilde \delta)=(k+1)\delta, \quad [s]^{-}(\tilde \delta)=k\tilde \delta\quad \text{ if } \quad k\tilde \delta\leq s<(k+1)\tilde \delta.$$
Several terms, previously indexed over $n$ in \cref{subsec:wong-zakai}, will now be indexed over $\delta$, e.g. we will write $ X^\delta$, $W^\delta$, $f_\delta$ and $s(t,\delta)$ in place of $X^n$, $W^n$, $f_n$ and $s(t,n)$
Finally, throughout this appendix we will adopt the convention that, whenever the symbol $\lesssim$ appears in the estimates, the hidden constant may depend on $T$, $\max_{ij} |c_{ij}|$ or $d$, but no other parameters.

Before going to the actual proof of \cref{smoothdrift}, we need to collect a few basic observations coming from that of \cite[IV. Lemma 7.1]{IkedaWatanabe}.
Therein it is shown that, for any $k\in \mathbb{N}$, it holds
\[
k\delta s_{ij}(k\delta,\delta)=k\delta s_{ij}(\delta,\delta)
\]
and moreover that
\begin{equation}\label{eq:useful-identity}
s_{ij}(\delta,\delta)-c_{ij}(k\delta,\delta)+\frac{1}{2}\delta_{ij}
= \frac{1}{k \delta} \mE\big[W^{\delta}_{0,i} (0) W_{k\delta,j} - W^{\delta,}_{0,i} W^{\delta}_{0,j}\big]; 
\end{equation}
for the choice $k=N(\delta)$, this readily implies the estimate
\begin{align*}
    \big|s_{ij}(\delta,\delta)-c_{ij}(\tilde\delta,\delta)+\frac{1}{2}\delta_{ij}\big|
    \lesssim \tilde{\delta}^{-1} [\tilde \delta^{\frac{1}{2}} \delta^{\frac{1}{2}} + \delta] \lesssim N(\delta)^{-\frac{1}{2}}
\end{align*}
which combined with \cref{asssc} implies
\begin{equation}\label{eq:estim-cij}
    |c_{ij}(\tilde\delta,\delta)-c_{ij}|^2 \lesssim f_\delta^2 + N(\delta)^{-1}.
\end{equation}
We can now proceed with the

\begin{proof}[Proof of \cref{smoothdrift}]
First by \cite[IV. 7 (7.47)]{IkedaWatanabe}, for every $s\leq t$ we have
\begin{align}\label{Xn-t-s}
|X^\delta_t-X_s^\delta|\leq \Vert b\Vert_\infty(t-s) + \Vert \sigma\Vert_\infty \int_s^t | \dot W_s^\delta|\dd s.
\end{align}
 As in \cite[IV. 7 (7.48)]{IkedaWatanabe}, we decompose $X^n-X$ into the following four terms:
\begin{align}
     \label{4-terms-dcom}
     X_t^\delta-X_t:=H_1(t)+H_2(t)+H_3(t)+H_4(t)
\end{align}
where 
\begin{align*} 
     H_1(t) =\int_{[t]^{-}(\tilde \delta)}^t\sigma(X_s^\delta)\dd W_s^\delta & 
     -\int_{[t]^{-}(\tilde \delta)}^t\sigma(X_s)\dd W_s\nonumber 
     - \int_{[t]^{-}(\tilde \delta)}^t (c:\sigma\cdot\nabla \sigma)(X_s)\dd s,
\end{align*}
\begin{align*} 
     H_2(t)=\int^{[t]^{-}(\tilde \delta)}_{\tilde \delta}\sigma(X_s^\delta)\dd W_s^\delta
     &-\int^{[t]^{-}(\tilde \delta)}_{\tilde \delta}\sigma(X_s)\dd W_s\nonumber
     - \int^{[t]^{-}(\tilde \delta)}_{\tilde \delta} (c:\sigma\cdot\nabla \sigma)(X_s)\dd s,
\end{align*}
\begin{align*} 
         H_3(t)=\int_0^{\tilde \delta}\sigma(X_s^\delta)\dd W_s^\delta
         &-\int_0^{\tilde \delta}\sigma(X_s)\dd W_s\nonumber
         - \int_0^{\tilde \delta} (c:\sigma\cdot\nabla\sigma)(X_s)\dd s
\end{align*}
 and 
\begin{align*} 
    H_4(t)=\int_0^tb(X^\delta_s)\dd s-\int_0^tb(X_s)\dd s.
\end{align*}
Easily observe that 
\begin{align}
    \label{diff-H4}
    \mE\Big[\sup_{0\leq t\leq T}|H_4(t)|^2\Big]\leq \Vert \nabla b\Vert_\infty^2\int_0^T\mE[|X^\delta_s-X_s|^2]\dd s.
\end{align}
We rewrite 
$$H_1(t)=H_{11}(t)-H_{12}(t)-H_{13}(t).$$
Following from  \cite[IV. 7 (7.54)]{IkedaWatanabe} we get 
\begin{align*}
     \mE\Big[\sup_{0\leq t\leq T}|H_{13}(t)|^2\Big]
     \leq \Vert\sigma\cdot\nabla\sigma\Vert^2_\infty\, N(\delta)^2 \delta^2,
\end{align*}
 and \cite[IV. 7 (7.55)]{IkedaWatanabe} gives us
 \begin{align*}
      \mE\Big[\sup_{0\leq t\leq T}|H_{11}(t)|^2\Big]
      \lesssim \Vert\sigma\Vert^2_\infty\, N(\delta)^{\frac{3}{2}}\, \delta^{\frac{1}{2}}.
 \end{align*}
 Note that 
 \begin{align*}
     H_{12}=\sigma(X_{[t]^-(\tilde \delta)})(W_t-W_{[t]^-(\tilde \delta)})+\int_{[t]^{-}(\tilde \delta)}^t[\sigma(X_s)-\sigma(X_{[t]^-(\tilde \delta)})]\dd W_s=:H_{121}(t)+H_{122}(t).
 \end{align*}
 Then \cite[IV. 7 (7.56)]{IkedaWatanabe} yields 
 \begin{align*}
      \mE\Big[\sup_{0\leq t\leq T}|H_{121}(t)|^2\Big]
      \lesssim \Vert\sigma\Vert^2_\infty \, N(\delta)^{\frac{1}{2}}\, \delta^{\frac{1}{2}} 
 \end{align*}
 and \cite[IV. 7 (7.57)]{IkedaWatanabe} implies
  \begin{align*}
      \mE\Big[\sup_{0\leq t\leq T}|H_{122}(t)|^2\Big]
      \lesssim \Vert\nabla\sigma\Vert^2_\infty (\Vert\sigma\Vert^2_\infty+\Vert b\Vert_\infty^2)\, \tilde\delta,
 \end{align*}
 where inside these passages we used the inequality
 $$\mathbb{E} \big[ |X_s(x)-x|^2\big]\lesssim \Vert\sigma\Vert^2_\infty\, s+\Vert b\Vert_\infty^2\, s^2.$$
 Putting all these estimates together we get 
 \begin{align}\label{diff-H1}
    \mE\Big[\sup_{0\leq t\leq T}|H_1(t)|^2\Big]
    \lesssim \big(1 + \Vert\sigma\Vert^4_{C^1_x} + \| b\|_{\infty}^4\big) 
    \, N(\delta)^{\frac{3}{2}}\, \delta^{\frac{1}{2}}.
\end{align}
Since $|H_3|\leq \sup_{0\leq t\leq T}|H_1(t)|$, 
 \begin{align}
    \label{diff-H3}
    \mE\Big[\sup_{0\leq t\leq T}|H_3(t)|^2\Big]
    \leq  \mE\Big[\sup_{0\leq t\leq T}|H_1(t)|^2\Big]
    \lesssim
    \big(1 + \Vert\sigma\Vert^4_{C^1_x} + \| b\|_{\infty}^4\big) \, N(\delta)^{\frac{3}{2}}\, \delta^{\frac{1}{2}}.
\end{align}
We are only left with estimating $H_2$. By integration by parts, we have 
\begin{align*}
  \int^{(k+1)\tilde \delta}_{k\tilde \delta}\sigma(X_s^\delta)\dd W_s^\delta=&\sigma(X^\delta_{k\tilde \delta})[W^\delta_{(k+1)\tilde\delta}-W_{k\tilde \delta}^\delta]\\&+   \int^{(k+1)\tilde \delta}_{k\tilde \delta} \nabla \sigma(X_s^\delta)[\sigma(X_s^\delta)\dot W_s^\delta+b(X_s^\delta)]\cdot[W_{(k+1)\tilde \delta}^\delta-W_{k\tilde \delta}^\delta]\dd s
  \\=:&J_1(k)+J_2(k).
\end{align*}
Also \begin{align*}
    J_1(k)=&\sigma(X_{k\tilde \delta-\delta}^\delta)(W_{(k+1)\tilde \delta}-W_{k\tilde \delta})+[\sigma(X_{k\tilde \delta}^\delta)-\sigma(X_{k\tilde \delta-\delta}^\delta)](W^\delta_{(k+1)\tilde\delta}-W^\delta_{k\tilde\delta})
    \\&+\sigma(X_{k\tilde \delta-\delta}^\delta)(W^\delta_{(k+1)\tilde\delta}-W_{(k+1)\tilde\delta})+\sigma(X_{k\tilde \delta-\delta}^\delta)(W_{k\tilde\delta}-W_{k\tilde\delta}^\delta)
    \\=&:J_{11}(k)+J_{12}(k)+J_{13}(k)+J_{14}(k).
\end{align*}
Write 
$$H_2(t)=I_1(t)+I_2(t)+I_3(t)+I_4(t)+I_5(t)$$
where 
\begin{align*}
    I_1(t)=\sum_{k=1}^{m(t)-1}J_{11}(k)+\int_{\tilde\delta}^{[t]^-(\tilde\delta)}\sigma(X_s)\dd W_s,
\end{align*}
with $m(t)$ defined by $m(t):= \lfloor t/\tilde \delta\rfloor$, and
\begin{align*}
    I_i(t)=\sum_{k=1}^{m(t)-1}J_{1i}(k),\quad i=2,3,4,
\end{align*}
\begin{align*}
    I_5(t)=\sum_{k=1}^{m(t)-1}J_{2}(k)-\int_{\tilde\delta}^{[t]^-(\tilde\delta)}(c:\sigma\cdot\nabla\sigma)(X_s)\dd s.
\end{align*}
From  \cite[IV. 7 (7.63)]{IkedaWatanabe} we have 
\begin{align*}
 \mE\Big[\sup_{0\leq t\leq T}|I_1(t)|^2\Big]
 \lesssim \| \nabla \sigma\|_{\infty}^2 \int_0^T\mE[|X^\delta_s-X_s|^2]\dd s+(\Vert \sigma\Vert^2_\infty+\Vert b\Vert_\infty^2)\, N(\delta)^2\delta 
\end{align*}
and \cite[IV. 7 (7.64)]{IkedaWatanabe} shows
\begin{align*}
 \mE\Big[\sup_{0\leq t\leq T}|I_2(t)|^2\Big]
 \lesssim \Vert\nabla\sigma\Vert^2_\infty\, N(\delta)^{-1}. 
\end{align*}
 \cite[IV. 7 (7.65), (7.66)]{IkedaWatanabe} give us
 \begin{align*}
 \mE\Big[\sup_{0\leq t\leq T}|I_3(t)|^2\Big] + \mE\Big[\sup_{0\leq t\leq T}|I_4(t)|^2\Big]
 \lesssim \Vert\sigma\Vert^2_\infty\, N(\delta)^{-1}
\end{align*}
For $I_5$ we rewrite
$$I_5(t)=I_{51}(t)+I_{52}(t)+I_{53}(t)+I_{54}(t)+I_{55}(t)$$
with
\begin{align*}
 I_{51}^{(h)}(t)&=\sum_{k=1}^{m(t)-1} \int^{(k+1)\tilde \delta}_{k\tilde \delta}\sum_{i,l,j=1}^d[(\sigma_{il}\partial_l \sigma_{jh})(X_s^\delta)-(\sigma_{il}\partial_l \sigma_{jh})(X_{k\tilde\delta}^\delta)] \dot W_{s,i}^{\delta} [W_{(k+1)\tilde\delta,j}^{\delta}-W_{s,j}^{\delta}]\dd s,\\
 I_{52}^{(h)}(t)&=\sum_{k=1}^{m(t)-1} \int^{(k+1)\tilde \delta}_{k\tilde \delta}\sum_{j,l=1}^d(b_{l}\partial_l \sigma_{jh})(X_s^\delta)[W_{(k+1)\tilde\delta,j}^{\delta}-W_{s,i}^\delta]\dd s,
\\
 I_{53}^{(h)}(t)&=\sum_{k=1}^{m(t)-1} \int^{(k+1)\tilde \delta}_{k\tilde \delta}\sum_{i,l,j=1}^d[(\sigma_{il}\partial_l \sigma_{jh})(X_{k\tilde\delta}^\delta)[\dot W_{s,i}^\delta(W_{(k+1)\tilde\delta,j}^\delta-W_{s,j}^\delta)-c_{ij}(\tilde \delta,\delta)]\dd s,\\
 I_{54}^{(h)}(t)&=\sum_{k=1}^{m(t)-1} \int^{(k+1)\tilde \delta}_{k\tilde \delta}\sum_{i,l,j=1}^dc_{ij}[(\sigma_{il}\partial_l \sigma_{jh})(X_{k\tilde\delta}^\delta)-(\sigma_{il}\partial_l \sigma_{jh})(X_s)]\dd s,\\
  I_{55}^{(h)}(t)&=\sum_{k=1}^{m(t)-1}\sum_{i,l,j=1}^d[\tilde\delta(\sigma_{il}\partial_l \sigma_{jh}))(X_{k\tilde\delta}^\delta)-c_{ij}]. 
\end{align*}
First we get from  \cite[IV. 7 (7.67)]{IkedaWatanabe}
\begin{align*}
  \mE\Big[\sup_{0\leq t\leq T}|I_{55}(t)|^2\Big]
  \lesssim \sum_{i,j=1}^d\Vert \sigma\cdot\nabla\sigma\Vert^2_\infty (c_{ij}(\tilde \delta,\delta)-c_{ij})^2
  \end{align*}
  and by \cite[IV. 7 (7.68)]{IkedaWatanabe}
  \begin{align*}
  \mE\Big[\sup_{0\leq t\leq T}|I_{54}(t)|^2\Big]
  \lesssim \Vert\nabla(\sigma\cdot\nabla\sigma)\Vert^2_\infty \Big[\int_0^T\mE[|X^\delta_s-X_s|^2]\dd s+(\Vert \sigma\Vert^2_\infty +\Vert b\Vert^2_\infty)\, N(\delta)^2\delta\Big].\end{align*}
  Furthermore by  \cite[IV. 7 (7.69)]{IkedaWatanabe}
\begin{align*}
  \mE\Big[\sup_{0\leq t\leq T}|I_{52}(t)|^2\Big]
  \lesssim \Vert b\cdot\nabla\sigma\Vert^2_\infty \, N(\delta)^2\delta
  \end{align*}
  and for $I_{51}$ \cite[IV. 7 (7.70)]{IkedaWatanabe} yields
  \begin{align*}
  \mE\Big[\sup_{0\leq t\leq T}|I_{51}(t)|^2\Big]
  \lesssim (\Vert \sigma\Vert^2_\infty+\Vert b\Vert_\infty^2)\, N(\delta)^4\delta.\end{align*}
  \cite[IV. 7 (7.71)]{IkedaWatanabe} shows that
   \begin{align*}
  \mE\Big[\sup_{0\leq t\leq T}|I_{53}(t)|^2\Big]
  \lesssim \Vert\sigma\Vert^2_\infty\, \Big(N(\delta)^3\delta+( \sum_{i,j=1}^dN(\delta)^{-1} c_{ij}^*(\tilde \delta,\delta))^2+N(\delta)^{-1} \Big), \end{align*}
  for $c_{ij}^\ast$ defined by
  \begin{align*}
    c_{ij}^*(\tilde \delta,\delta):=\, N(\delta)\,c_{ij}(N(\delta)\delta,\delta)-(N(\delta)-1)c_{ij}((N(\delta)-1)\delta,\delta); 
  \end{align*}
  Applying identity \eqref{eq:useful-identity} for the choices $k=N(\delta)$, $k=N(\delta)-1$, one can see that
\begin{align*}
      | N(\delta)^{-1} c^\ast_{ij}(\tilde\delta,\delta)|
      & = \Big| N(\delta)^{-1} \Big[ s_{ij} (\delta,\delta) + \frac{1}{2} \delta_{ij} \Big] -(\tilde\delta)^{-1} \mE\Big[ W^{\delta}_{0,i}\, \big(W_{N(\delta)\delta,j}-W_{(N(\delta)-1)\delta,j}\big)\Big] \Big|\\
      & \lesssim N(\delta)^{-1} + (\tilde\delta)^{-1} \delta^{\frac{1}{2}} \delta^{\frac{1}{2}} = N(\delta)^{-1}.
\end{align*}
  Therefore
   \begin{align*}
  \mE\Big[\sup_{0\leq t\leq T}|I_{53}(t)|^2\Big]
  \lesssim  \Vert\sigma\Vert^2_\infty
  \big(N(\delta)^3 \delta +
  N(\delta)^{-1}\big), \end{align*}
  Collecting all the terms related to $I_{5}$ and applying \eqref{eq:estim-cij}, we get the following:
    \begin{align*}
  \mE\Big[\sup_{0\leq t\leq T}|I_{5}(t)|^2\Big]
  \lesssim & \Vert\nabla(\sigma\cdot\nabla\sigma)\Vert^2_\infty
  \int_0^T\mE[|X^\delta_s-X_s|^2 ]\dd s
  \nonumber\\&+\Big(1 + \Vert \sigma\Vert^8_{C^2_x}+\Vert b\Vert^4_\infty\Big) \Big( f_\delta^2 + N(\delta)^4\delta +
  N(\delta)^{-1}\Big).   \end{align*}
  Therefore
  \begin{align}\label{H2-diff}
   \mE\Big[\sup_{0\leq t\leq T}|H_{2}(t)|^2\Big]  \lesssim
   &\Big( \| \nabla \sigma\|_{\infty}^2 + \Vert\nabla(\sigma\cdot\nabla\sigma)\Vert^2_\infty \Big) \int_0^T\mE[|X^\delta_s-X_s|^2]\dd s
  \nonumber\\&+\Big(1 + \Vert \sigma\Vert^8_{C^2_x}+\Vert b\Vert^4_\infty\Big) \Big( f_\delta^2 + N(\delta)^4\delta +
  N(\delta)^{-1}\Big).
  \end{align}
  By combining \eqref{4-terms-dcom}, \eqref{diff-H4}, \eqref{diff-H1}, \eqref{diff-H3} and \eqref{H2-diff} we can conclude  that
  \begin{align*}
   \mE\Big[\sup_{0\leq t\leq T}|X_t-X_t^\delta|^2\Big]\lesssim
       &(1 + \Vert b\Vert_{C^1_x}^2 + \| \sigma\|_{C^1_x}^4 + \| \sigma\|_{\infty} \| \sigma\|_{C^2_x}^2)
   \int_0^T\mE[|X^\delta_s-X_s|^2]\dd s \nonumber\\&
  +\Big(1 + \Vert \sigma\Vert^8_{C^2_x}+\Vert b\Vert^4_\infty\Big) \Big( f_\delta^2 + N(\delta)^{\frac{3}{2}} \delta^{\frac{1}{2}} + N(\delta)^4\delta + 
  N(\delta)^{-1}\Big).
  \end{align*}
%
Gr\"onwall's inequality then implies the estimate
\begin{equation}\begin{split}\label{eq:almost-final}
    \mE\Big[\sup_{0\leq t\leq T}|X_t-X_t^\delta|^2\Big]
    & \leq \exp\Big( C' (1 + \Vert b\Vert_{C^1_x}^2 + \| \sigma\|_{C^1_x}^4 + \| \sigma\|_{\infty} \| \sigma\|_{C^2_x}^2) \Big)\times\\
    & \times \Big(1 + \Vert \sigma\Vert^8_{C^2_x}+\Vert b\Vert^4_\infty\Big) \Big( f_\delta^2 + N(\delta)^{\frac{3}{2}} \delta^{\frac{1}{2}} + N(\delta)^4\delta + N(\delta)^{-1}\Big).
\end{split}\end{equation}
Observe that, up to relabelling the constant $C'$, the polynomial term $1 + \Vert \sigma\Vert^8_{C^2_x}+\Vert b\Vert^4_\infty$ in \eqref{eq:almost-final} can be reabsorbed in the exponential.
Moreover recall that $\delta=n^{-1}$ and $N(\delta)\sim n^{\eps/4}$, so that
\[
N(\delta)^{\frac{3}{2}} \delta^{\frac{1}{2}} + N(\delta)^4\delta + N(\delta)^{-1}
\sim n^{\frac{3\eps}{8}-\frac{1}{2}} + n^{\eps-1} + n^{-\frac{\eps}{4}};
\]
optimizing it w.r.t. $\eps$ yields the choice $\eps=4/5$ and a corresponding decay rate $n^{-1/5}$. Combining these two last observations with \eqref{eq:almost-final} yields the desired estimate \eqref{errorest}.
\end{proof}

\section*{Acknowledgements}
LG is funded by the DFG under Germany's Excellence Strategy - GZ 2047/1, project-id 390685813.
CL acknowledges financial support by the DFG via Research Unit FOR 2402.
\bibliography{myBiblio}{}

\begin{thebibliography}{10}

\bibitem{bass2003brownian}
Richard~F. Bass and Zhen-Qing Chen.
\newblock Brownian motion with singular drift.
\newblock {\em The Annals of Probability}, 31(2):791--817, 2003.

\bibitem{caravenna2021}
Laura Caravenna and Gianluca Crippa.
\newblock A directional {L}ipschitz extension lemma, with applications to
  uniqueness and {L}agrangianity for the continuity equation.
\newblock {\em Comm. Partial Differential Equations}, 46(8):1488--1520, 2021.

\bibitem{catellier2016averaging}
R{\'e}mi Catellier and Massimiliano Gubinelli.
\newblock Averaging along irregular curves and regularisation of {ODE}s.
\newblock {\em Stochastic Processes and their Applications}, 126(8):2323--2366,
  2016.

\bibitem{de2019numerical}
Tiziano De~Angelis, Maximilien Germain, and Elena Issoglio.
\newblock A numerical scheme for stochastic differential equations with
  distributional drift.
\newblock {\em arXiv preprint arXiv:1906.11026}, 2019.

\bibitem{flandoli2010well}
Franco Flandoli, Massimiliano Gubinelli, and Enrico Priola.
\newblock Well-posedness of the transport equation by stochastic perturbation.
\newblock {\em Inventiones mathematicae}, 180(1):1--53, 2010.

\bibitem{flandoli2017multidimensional}
Franco Flandoli, Elena Issoglio, and Francesco Russo.
\newblock Multidimensional stochastic differential equations with
  distributional drift.
\newblock {\em Transactions of the American Mathematical Society},
  369(3):1665--1688, 2017.

\bibitem{galeati2022continuous}
Lucio Galeati, Fabian~A. Harang, and Avi Mayorcas.
\newblock Distribution dependent {SDE}s driven by additive continuous noise.
\newblock {\em Electronic Journal of Probability}, 27:1--38, 2022.

\bibitem{galeati2022distribution}
Lucio Galeati, Fabian~A. Harang, and Avi Mayorcas.
\newblock Distribution dependent {SDE}s driven by additive fractional brownian
  motion.
\newblock {\em Probability Theory and Related Fields}, pages 1--59, 2022.

\bibitem{geiss2022concave}
Sarah Geiss.
\newblock Concave and other generalizations of stochastic {G}ronwall
  inequalities.
\newblock {\em arXiv preprint arXiv:2204.06042}, 2022.

\bibitem{hao2022strong}
Zimo Hao, Michael R{\"o}ckner, and Xicheng Zhang.
\newblock Strong convergence of propagation of chaos for {M}c{K}ean--{V}lasov
  {SDE}s with singular interactions.
\newblock {\em arXiv preprint arXiv:2204.07952}, 2022.

\bibitem{hu2017nonlinear}
Yaozhong Hu and Khoa L{\^e}.
\newblock Nonlinear {Y}oung integrals and differential systems in {H}{\"o}lder
  media.
\newblock {\em Transactions of the American Mathematical Society},
  369(3):1935--2002, 2017.

\bibitem{huang2021distribution}
Xing Huang, Panpan Ren, and Feng-Yu Wang.
\newblock Distribution dependent stochastic differential equations.
\newblock {\em Frontiers of Mathematics in China}, 16(2):257--301, 2021.

\bibitem{huang2019distribution}
Xing Huang and Feng-Yu Wang.
\newblock Distribution dependent {SDE}s with singular coefficients.
\newblock {\em Stochastic Processes and their Applications},
  129(11):4747--4770, 2019.

\bibitem{huang2020mckean}
Xing Huang and Feng-Yu Wang.
\newblock Mc{K}ean--{V}lasov {SDE}s with drifts discontinuous under
  {W}asserstein distance.
\newblock {\em arXiv preprint arXiv:2002.06877}, 2020.

\bibitem{IkedaWatanabe}
Noboyuki Ikeda and Shinzo Watanabe.
\newblock {\em Stochastic Differential Equations and Diffusion Processes}.
\newblock NorthHolland Publishing Company, Amsterdam Oxford New York, 1981.

\bibitem{issoglio2021mckean}
Elena Issoglio and Francesco Russo.
\newblock Mc{K}ean {SDE}s with singular coefficients.
\newblock {\em arXiv preprint arXiv:2107.14453}, 2021.

\bibitem{kohatsu2017weak}
Arturo Kohatsu-Higa, Antoine Lejay, and Kazuhiro Yasuda.
\newblock Weak rate of convergence of the {E}uler--{M}aruyama scheme for
  stochastic differential equations with non-regular drift.
\newblock {\em Journal of Computational and Applied Mathematics}, 326:138--158,
  2017.

\bibitem{lacker2018strong}
Daniel Lacker.
\newblock On a strong form of propagation of chaos for {M}c{K}ean--{V}lasov
  equations.
\newblock {\em Electronic Communications in Probability}, 23:1--11, 2018.

\bibitem{le2021taming}
Khoa L{\^e} and Chengcheng Ling.
\newblock Taming singular stochastic differential equations: {A} numerical
  method.
\newblock {\em arXiv preprint arXiv:2110.01343}, 2021.

\bibitem{ling2021wong}
Chengcheng Ling, Sebastian Riedel, and Michael Scheutzow.
\newblock A {W}ong-{Z}akai theorem for {SDE}s with singular drift.
\newblock {\em arXiv preprint arXiv:2109.12158}, 2021.

\bibitem{ling2021strong}
Chengcheng Ling and Longjie Xie.
\newblock Strong solutions of stochastic differential equations with
  coefficients in mixed-norm spaces.
\newblock {\em Potential Analysis}, pages 1--15, 2021.

\bibitem{mishura2020existence}
Yuliya Mishura and Alexander Veretennikov.
\newblock Existence and uniqueness theorems for solutions of
  {M}c{K}ean--{V}lasov stochastic equations.
\newblock {\em Theory of Probability and Mathematical Statistics}, 2020.

\bibitem{ren2022linearization}
Panpan Ren, Michael R{\"o}ckner, and Feng-Yu Wang.
\newblock Linearization of nonlinear {F}okker--{P}lanck equations and
  applications.
\newblock {\em Journal of Differential Equations}, 322:1--37, 2022.

\bibitem{rockner2021well}
Michael R{\"o}ckner and Xicheng Zhang.
\newblock Well-posedness of distribution dependent {SDE}s with singular drifts.
\newblock {\em Bernoulli}, 27(2):1131--1158, 2021.

\bibitem{simon1986compact}
Jacques Simon.
\newblock Compact sets in the space $l^p(o,t;b)$.
\newblock {\em Annali di Matematica pura ed applicata}, 146(1):65--96, 1986.

\bibitem{siorpaes2018}
Pietro Siorpaes.
\newblock Applications of pathwise {B}urkholder--{D}avis--{G}undy inequalities.
\newblock {\em Bernoulli}, 24(4B):3222--3245, 2018.

\bibitem{stein1970singular}
Elias~M. Stein.
\newblock {\em Singular integrals and differentiability properties of
  functions}, volume~2.
\newblock Princeton university press, 1970.

\bibitem{villani2009}
C{\'e}dric Villani.
\newblock {\em Optimal transport: old and new}, volume 338.
\newblock Springer, 2009.

\bibitem{XXZZ2020}
Pengcheng Xia, Longjie Xie, Xicheng Zhang, and Guohuan Zhao.
\newblock {$L^q(L^p)$}-theory of stochastic differential equations.
\newblock {\em Stochastic Process. Appl.}, 130(8):5188--5211, 2020.

\bibitem{XieZhang}
Longjie Xie and Xicheng Zhang.
\newblock Ergodicity of stochastic differential equations with jumps and
  singular coefficients.
\newblock {\em Ann. Inst. H. Poincar\'e Probab. Statist.}, 56:175--229, 2020.

\bibitem{zhang2021zvonkin}
Shao-Qin Zhang and Chenggui Yuan.
\newblock A {Z}vonkin's transformation for stochastic differential equations
  with singular drift and applications.
\newblock {\em Journal of Differential Equations}, 297:277--319, 2021.

\bibitem{zhang2016stochastic}
Xicheng Zhang.
\newblock Stochastic differential equations with {S}obolev diffusion and
  singular drift and applications.
\newblock {\em Ann. Appl. Probab.}, 26(5):2697--2732, 2016.

\bibitem{zhang2017heat}
Xicheng Zhang and Guohuan Zhao.
\newblock Heat kernel and ergodicity of {SDE}s with distributional drifts.
\newblock {\em arXiv preprint arXiv:1710.10537}, 2017.

\bibitem{zhao2020distribution}
Guohuan Zhao.
\newblock On distribution dependent {SDE}s with singular drifts.
\newblock {\em arXiv preprint arXiv:2003.04829}, 2020.

\end{thebibliography}
\bibliographystyle{plain}

\end{document}